\documentclass[12pt,letterpaper]{amsart}

\usepackage[english]{babel}
\usepackage[utf8]{inputenc}
\usepackage{amsmath,amssymb,amsthm}

\usepackage{bbm,bm}
\usepackage{ushort}
\usepackage{enumitem}
\usepackage{calligra,mathrsfs} 
\usepackage[marginparwidth=2cm, marginparsep=1cm]{geometry} 

\usepackage{latexsym,mathtools,braket}
\usepackage{tikz-cd}
\usepackage{thmtools,thm-restate}

\usepackage{lineno}
\usepackage{marginnote}

\newcounter{marginnote}

\usepackage[colorlinks,plainpages,backref]{hyperref}
\usepackage{cleveref}

\numberwithin{equation}{subsection}
\numberwithin{figure}{section}

\declaretheorem[name=Theorem,
	refname={theorem,theorems},
	Refname={Theorem,Theorems},
	numberwithin=section]{theorem}
\declaretheorem[name=Proposition,
	refname={proposition, propositions},
	Refname={Proposition, Propositions},
	sibling=theorem]{proposition}
\declaretheorem[name=Lemma,
	refname={lemma,lemmas},
	Refname={Lemma,Lemmas},
	sibling=theorem]{lemma}
\declaretheorem[name=Corollary,
	refname={corollary,corollaries},
	Refname={Corollary,Corollaries},
	sibling=theorem]{corollary}

\declaretheorem[name=Remark,
	refname={remark,remarks},
	Refname={Remark,Remarks},
	style=remark,
	sibling=theorem]{remark}
\declaretheorem[name=Example,
	refname={example,examples},
	Refname={Example,Examples},
	sibling=theorem,
	style=remark]{example}
\declaretheorem[name=Definition,
	refname={definition,definitions},
	Refname={Definition,Definitions},
	style=remark,
	sibling=theorem]{definition}

\newcommand{\CC}{\mathbb{C}}
\newcommand{\ZZ}{\mathbb{Z}}
\newcommand{\NN}{\mathbb{N}}
\newcommand{\RR}{\mathbb{R}}
\newcommand{\QQ}{\mathbb{Q}}
\newcommand{\OO}{\mathcal{O}}
\renewcommand{\vec}[1]{\mathbf{#1}}
\newcommand{\Bigwedge}{\textstyle\bigwedge\!}
\DeclareMathOperator{\sRes}{sRes}

\DeclareMathOperator{\Spec}{Spec}
\DeclareMathOperator{\Supp}{Supp}

\DeclareMathOperator{\Hom}{Hom}
\DeclareMathOperator{\Mod}{Mod}

\newcommand{\Lder}{\mathrm{L}}
\newcommand{\Rder}{\mathrm{R}}
\DeclareMathOperator{\RGamma}{R\Gamma}

\newcommand{\HH}{H}
\newcommand{\RuHom}{\mathrm{R}\!\uHom}
\DeclareMathOperator{\DD}{D}

\DeclareMathOperator{\Tot}{Tot}
\DeclareMathOperator{\cone}{cone}
\newcommand{\qi}{\stackrel{\mathrm{qi}}{=}}

\DeclareMathOperator{\uHom}{\underline{Hom}}

\newcommand{\Rdual}{\mathbf{D}}
\DeclareMathOperator{\qdeg}{qdeg}
\DeclareMathOperator{\tdeg}{tdeg}
\DeclareMathOperator{\orbit}{O}

\newcommand{\Ddual}{\mathbb{D}}
\DeclareMathOperator{\FL}{FL}
\DeclareMathOperator{\fSupp}{fSupp}
\DeclareMathOperator{\cofSupp}{cofSupp}

\DeclareMathOperator{\relint}{relint}

\DeclareMathOperator{\shHom}{\mathscr{H}\text{\kern -3pt {\calligra\large om}}\,}
\DeclareMathOperator{\shRHom}{R\mathscr{H}\text{\kern -3pt {\calligra\large om}}\,}
\newcommand{\an}{\mathrm{an}}

\begin{document}
\title{$A$-Hypergeometric Modules and Gauss--Manin Systems}

\author{Avi Steiner}
\address{A.~Steiner\\
  Purdue University\\
  Dept.\ of Mathematics\\
  150 N.\ University St.\\
  West Lafayette, IN 47907\\ USA}
\email{steinea@purdue.edu}
\thanks{Supported by the National Science Foundation under grant DMS-1401392.}
\date{\today}

\begin{abstract}
	Let $A$ be a $d\times n$ integer matrix. Gel$'$fand et al.\ proved that most $A$-hypergeometric systems have an interpretation as a Fourier--Laplace transform of a direct image. The set of parameters for which this happens was later identified by Schulze and Walther as the set of not strongly resonant parameters of $A$. A similar statement relating $A$-hypergeometric systems to exceptional direct images was proved by Reichelt. In this article, we consider a hybrid approach involving neighborhoods $U$ of the torus of $A$ and consider compositions of direct and exceptional direct images. Our main results characterize for which parameters the associated $A$-hypergeometric system is the inverse Fourier--Laplace transform of such a ``mixed Gauss--Manin'' system.
	
	In order to describe which $U$ work for such a parameter, we introduce the notions of fiber support and cofiber support of a $D$-module.
	
	If the semigroup ring $\CC[\NN{A}]$ is normal, we show that every $A$-hypergeometric system is ``mixed Gauss--Manin''. We also give an explicit description of the neighborhoods $U$ which work for each parameter in terms of primitive integral support functions.	
\end{abstract}

\maketitle



\section{Introduction}

Let $A\in \ZZ^{d\times n}$ be an integer matrix with columns $\vec{a}_1,\ldots,\vec{a}_n$ such that $\ZZ A=\ZZ^d$. Assume that $\NN A$ is pointed, i.e.~that $\NN{A}\cap -\NN{A}=0$. Define the following objects:
\begin{align*}
	S_A		&= \CC[\NN{A}]\text{, the semigroup ring of }A\\
	X_A 	&= \Spec{S_A}\text{, the toric variety of }A\\
	T_A		&= \Spec{\CC[\ZZ^d]}\text{, the torus of }A\\
	D_A		&= \CC[x_1,\ldots,x_n,\partial_1,\ldots,\partial_n]\text{, the }n\text{th Weyl algebra}
\end{align*} 
Associated to this data, Gel$'$fand, Graev, Kapranov, and Zelevinski\u{\i} defined in \cite{GGZ87,GZK89} 
a family of $D_A$-modules today referred to either as \emph{GKZ-} or \emph{$A$-hypergeometric} systems. These systems are defined as follows: Let $\beta\in \CC^d$. The \emph{Euler operators} of $A$ are the operators
\begin{equation}\label{eq:euler-ops}
	E_i \coloneqq a_{i1}x_1\partial_1+\cdots+a_{in}x_n\partial_n,\quad i=1,\ldots,d.
\end{equation}
The $A$-hypergeometric system corresponding to $\beta$ is then defined to be
\[ M_A(\beta)\coloneqq \frac{D_A}{\Braket{\partial^{\vec{u}_+}-\partial^{\vec{u}_-}|A\vec{u}=0, \vec{u}\in \ZZ^n} + \Braket{E_1-\beta_1,\ldots,E_d-\beta_d}},\]
where the brackets (here and throughout this paper) denote a left ideal.

\subsection{Torus Embeddings and Direct Images}
The torus embedding 
\begin{equation}\label{eq:torus-emb}
	\begin{gathered}
		\varphi\colon T_A\hookrightarrow \widehat{\CC^n}\coloneqq\Spec{\CC[\partial_1,\ldots,\partial_n]}\\
		t\mapsto (t^{\vec{a}_1},\ldots,t^{\vec{a}_n})
	\end{gathered}
\end{equation}
induces a closed immersion of $X_A$ into $\widehat{\CC^n}$. On the torus, the data $A$ and $\beta$ give a $D$-module 
\begin{equation*}
	\mathcal{O}_{T_A}^\beta \coloneqq \mathcal{O}_{T_A}t^{-\beta}.
\end{equation*}
A natural question is then whether and how this $D_{T_A}$-module is related to (the inverse Fourier--Laplace transform (see \S\ref{subsec:FL}) of) the $A$-hypergeometric system $M_A(\beta)$. A foundational result in this direction was given in \cite[Theorem~4.6]{GKZ90}: For non-resonant $\beta$, the Fourier--Laplace transform of the $D$-module direct image $\varphi_+\mathcal{O}_{T_A}^\beta$ is isomorphic to $M_A(\beta)$. This result was strengthened in \cite[Corollary~3.7]{SW09} to: the Fourier--Laplace transform of $\varphi_+\mathcal{O}_{T_A}^\beta$ is isomorphic to $M_A(\beta)$ if and only if $\beta$ is not in the set
\begin{equation}\label{eq:sres}
	\sRes(A)\coloneqq \bigcup_{j=1}^n \qdeg{H^1_{\braket{t^{\vec{a}_j}}}(S_A)}
\end{equation}
of \emph{strongly resonant} parameters. Here, $\qdeg$ denotes the set of quasidegrees of a $\ZZ^d$-graded module and is defined in \Cref{def:qdeg}. The $\ZZ^d$-grading on $S_A$ is defined in \S\ref{subsec:toric}.

It was then shown in \cite[Proposition~1.14]{Rei14} that for certain other $\beta$, the inverse Fourier--Laplace transform of $M_A(\beta)$ may be related to the $D$-module exceptional direct image $\varphi_\dagger\mathcal{O}_{T_A}^\beta$. Namely, $\varphi_\dagger\mathcal{O}_{T_A}^\beta \cong \FL^{-1}(M_A(\beta))$ if $A$ is homogeneous (i.e.~the vector $(1,\ldots,1)$ is in the rowspan of $A$), $\beta\in\QQ^d$, and $\beta$ is not in the set
\[ \bigcup_{F\text{ face of }A} \left[(\ZZ^d\cap \RR_{\geq0}A) + \CC{F}\right].\]

In \Cref{th:plus-dag,th:dag-plus}, we give simultaneous generalizations of both \cite[Corollary~3.3]{SW09} and \cite[Proposition~1.14]{Rei14}. These generalizations  allow (the inverse Fourier--Laplace transform of) more $A$-hypergeometric systems to be equipped with a mixed Hodge module structure. In a future paper, we will use the normal case of these generalizations (\Cref{th:normal}) to compute for normal $A$ the projection and restriction of $M_A(\beta)$ to coordinate subspaces of the form $\CC^F$, where $F$ is a face of $A$; and, if $A$ is in addition homogeneous, to show that the holonomic dual of $M_A(\beta)$ is itself $A$-hypergeometric.

\subsection{Main Idea}
Given a Zariski open subset $U\subseteq \widehat{\CC^n}$ containing $T_A$, write 
\[\iota_U\colon T_A\hookrightarrow U\]
for the embedding of $T_A$ into $U$ and 
\[\varpi_U\colon U\hookrightarrow \widehat{\CC^n}\] 
for the inclusion of $U$ into $\widehat{\CC^n}$. The first main result in this paper, \Cref{th:plus-dag}, provides an equivalent condition (in terms of the various local cohomology complexes $\RGamma_{\orbit(F)}(S_A)$ with supports in the orbit $\orbit(F)$; see \S\ref{subsec:toric} and \S\ref{subsec:LC}) for 
\[ K^A_\bullet(S_A; E_A-\beta) \cong \FL(\varpi_{U+}\iota_{U\dagger}\mathcal{O}_{T_A}^\beta)\]
for some such $U$,
while the second main result, \Cref{th:dag-plus}, does the same (this time in terms of the various localizations $S_A[\partial^{-F}]$) for
\[ K^A_\bullet(S_A; E_A-\beta) \cong \FL(\varpi_{U\dagger}\iota_{U+}\mathcal{O}_{T_A}^\beta).\]

The condition for the first main result has two parts: First is a requirement that $\beta$ not be rank-jumping. Second is a requirement about certain sets akin to Saito's $E_F(\beta)$ sets (see \Cref{def:E+dag,def:mgm}). Those parameters $\beta$ for which both these conditions hold are called \emph{dual mixed Gauss--Manin} (see \Cref{def:mgm}).

On the other hand, the condition for the second main result can be expressed as a requirement about Saito's $E_F(\beta)$ sets themselves. Those parameters $\beta$ for which this condition holds are called \emph{mixed Gauss--Manin} (see \Cref{def:mgm}).

The proof of \Cref{th:plus-dag} is accomplished as follows: First, we restate in terms of local cohomology via \Cref{lem:dag-plus}. Then, using the relationship between fiber support (\Cref{def:fiber-support}) and local cohomology in \Cref{prop:lc-vs-fSupp}, we focus in on the restriction to torus orbits. These restrictions are computed for general inverse-Fourier--Laplace-transformed Euler--Koszul complexes in \Cref{th:orbit-restr}.

We also use in the proof that $\varphi_\dagger\mathcal{O}_{T_A}^\beta$ can be expressed in two ways as an Euler--Koszul complex (see \Cref{def:Euler--Koszul}): As an Euler--Koszul complex of the dualizing complex of $S_A$ (\Cref{cor:phi-dagger}), and as an Euler--Koszul complex of $S_A$ itself (\Cref{prop:phi-dagger}).

The proof of \Cref{th:dag-plus} follows a similar route.

\subsection*{Acknowledgements}
\thanks{Support by the National Science Foundation under grant DMS-1401392 
is gratefully acknowledged. We would also like to thank Uli Walther for his support and guidance, Thomas Reichelt for asking the question which led to this paper, and Claude Sabbah and Kiyoshi Takeuchi for intriguing discussions.}

\section{Notation and Conventions}
Subsection \ref{subsec:toric} defines various symbols related to the affine semigroup $\NN{A}$. Subsection \ref{subsec:*things} recalls some common notions and facts about (multi-)graded rings and modules. Local cohomology with supports in a locally closed subset is recalled in subsection \ref{subsec:LC}. Conventions and notation relating to varieties, $D$-modules, sheaves, and derived categories are given in subsection \ref{subsec:other-convs} along with the definition of the Fourier-Laplace transform. Finally, in subsection \ref{subsec:Euler--Koszul}, we recall the notion of Euler--Koszul complexes.

\subsection{Toric and GKZ Conventions/Notation}\label{subsec:toric}
Let $R_A$ be the polynomial ring $\CC[\partial_1,\ldots, \partial_n]$, and set
\begin{equation}\label{eq:hatCCn}
	\widehat{\CC^n} \coloneqq \Spec R_A.
\end{equation}
This space is to be (loosely) interpreted as the ``Fourier--Laplace-transformed version'' of $\CC^n$, hence the $\widehat{\hphantom{X}}$ (cf.~\S\ref{subsec:FL}).

Let $I_A\subseteq R_A$ be the toric ideal corresponding to the embedding $\varphi$ from \eqref{eq:torus-emb}---we identify $S_A$ with the quotient $R_A/I_A$. The torus embedding also induces an action of $T_A$ on $\widehat{\CC^n}$, which in turn induces an action (the \emph{contragredient} action) of $T_A$ on $R_A$ via 
\[(t\cdot f)(\partial_1,\ldots,\partial_n)=f(t^{-\vec{a}_1}\partial_1,\ldots, t^{-\vec{a}_n}\partial_n).\]
An element $f\in R_A$ is \emph{homogeneous of degree $\alpha\in \ZZ^d$} if $t\cdot f = t^{-\alpha} f$ for all points $t\in (\CC^*)^n$; it is \emph{homogeneous} if it is homogeneous for some $\alpha$. In particular, $\deg(\partial_i)=\vec{a}_i$, and $S_A$ is a $\ZZ^d$-graded $R_A$-module.

Set 
\begin{equation}\label{eq:epsA}
	\varepsilon_A \coloneqq \vec{a}_1+\cdots + \vec{a}_n.
\end{equation}

Write $\hat{M}_A(\beta)$ for the inverse Fourier--Laplace transform (see \S\ref{subsec:FL})  of the GKZ system $M_A(\beta)$.

\subsubsection{Faces}
A submatrix $F$ of $A$ is called a \emph{face} of $A$, written $F\preceq A$, if $F$ has $d$ rows and $\RR_{\geq 0} F$ is a face of $\RR_{\geq 0} A$. Given ${F\preceq A}$, we make the following definitions:
\begin{equation}
	T_F\coloneqq \Spec\CC[\ZZ{F}]
\end{equation}
is the torus of $F$. The monomial in $\CC[\ZZ{F}]$ corresponding to $\alpha\in \ZZ{F}$ is written $t^\alpha$. Denote by
\begin{equation}\label{eq:orbit}
	\orbit(F)\coloneqq T_A\cdot \mathbbm{1}_F\subseteq \widehat{\CC^n}
\end{equation}
the orbit in $\widehat{\CC^n}$ corresponding to $F$ (where the $i$th coordinate of $\mathbbm{1}_F$ is $1$ if ${\vec{a}_i\in F}$ and $0$ otherwise). Note that the inclusion $\ZZ{F}\hookrightarrow \ZZ^d$ induces an isomorphism $\orbit(F)\cong T_F$. The rank of $F$ is denoted by $d_F$, and if $G\preceq A$ with $G\succeq F$, we set 
\begin{equation}
	d_{G/F}\coloneqq d_G-d_F.
\end{equation}
Define the ideal
\begin{equation}\label{eq:IAF}
	I^A_F\coloneqq I_A + \Braket{\partial_i|\vec{a}_i\notin F}
\end{equation}
of $R_A$, and set
\begin{equation}\label{eq:del-to-minusF}
	\partial^{kF}\coloneqq \prod_{\vec{a}_i\in F} \partial_i^k \qquad (k\in \ZZ).
\end{equation}

Given $u\in (\CC{F})^*\coloneqq \Hom_\CC(\CC{F},\CC)$, define $\vartheta_u$ to be the invariant vector field on $T_F$ defined by
\begin{equation}\label{eq:vartheta}
	\vartheta_u(t^\alpha) \coloneqq \braket{\alpha,u} t^\alpha \quad (\alpha\in\ZZ{F}),
\end{equation}
where $\braket{,}$ denotes the standard pairing of dual spaces. These vector fields span the Lie algebra of $T_F$; therefore, $D_{T_F}$ is generated as a $\CC$-algebra by $\mathcal{O}_{T_F}$ and the vector fields $\Set{\vartheta_u|u\in (\CC{F})^*}$ (both of these claims may be proven in a straightforward manner, e.g.~by choosing coordinates). 

For $\lambda\in \CC{F}$, define the $D_{T_F}$-module
\begin{equation}\label{eq:Gauss-Manin}
	\mathcal{O}_{T_F}^\lambda \coloneqq \mathcal{O}_{T_F}t^{-\lambda},
\end{equation}
where $t^{-\lambda}$ is a formal symbol subject to the $D_{T_F}$-action
\[\vartheta_u(ft^{-\lambda}) \coloneqq [\vartheta_u(f) - \braket{\lambda, u} f]t^{-\lambda} \qquad (u\in (\CC{F})^*).\]
This module is isomorphic to $\mathcal{O}_{T_F}$ as an $\mathcal{O}_{T_F}$-module and so is in particular an integrable connection. Moreover, it is a simple $D_{T_F}$-module.

\subsection{Graded Rings and Modules}\label{subsec:*things}
For more details about (multi-)graded rings and modules than are given here, refer to \cite{GW78,BH93,ms05}.

\subsubsection{Twists} Let $M$ be a graded module over a $\ZZ^k$-graded ring $R$. Given an $\alpha\in \ZZ^k$, define the graded module $M(\alpha)$ to be $M$ as an ungraded $R$-module and to have degree $\gamma$ component
\[ M(\alpha)_\gamma \coloneqq M_{\alpha+\gamma}.\]

\subsubsection{*- Properties}
A \emph{*-simple ring} is a graded ring with no homogeneous (two-sided) ideals. A graded module over a graded ring $S$ is \emph{*-free} if it is a direct sum of graded twists of $S$. A graded module over a graded ring $S$ is \emph{*-injective} if it is an injective object in the category of graded $S$-modules.

\subsubsection{(Weakly) $\NN{A}$-Closed Subsets}
As in \cite[p143]{ishida}, we make the following definitions: A subset $E$ of $\ZZ^d$ is \emph{$\NN{A}$-closed} if $E+\NN{A}\subseteq E$. If $E$ is $\NN{A}$-closed, define $\CC\{E\}$ to be the graded $S_A$-submodule of $\CC[\ZZ^d]\coloneqq \CC[t_1^\pm,\ldots,t_d^\pm]$
\begin{equation}
	\CC\{E\} \coloneqq \CC\Set{t^\alpha| \alpha\in E}
\end{equation}
given as the vector space spanned by $\Set{t^\alpha|\alpha \in E}$.

A subset $E$ of $\ZZ^d$ is \emph{weakly $\NN{A}$-closed} if $(E+\NN{A})\setminus E$ is $\NN{A}$-closed. If $E$ is weakly $\NN{A}$-closed, define
\begin{equation}\label{eq:wk-cld}
	\CC\{E\}\coloneqq \CC\{E+\NN{A}\}/\CC\{(E+\NN{A})\setminus E\}.
\end{equation}

\subsubsection{*-Injective Modules}\label{subsubseq:*-inj}
By \cite[Prop.~11.24]{ms05}, every indecomposable *-injective $S_A$-module is a $\ZZ^d$-graded twist of $\CC\{\NN{F}-\NN{A}\}$ for some face $F\preceq A$. Note that by \cite[Lem.~11.12 together with Prop.~11.24]{ms05}, $\CC\{\NN{F}-\NN{A}\}$ is the injective envelope of $S_F$ in the category of graded $S_A$-modules.

\subsubsection{Graded Hom}
Given graded modules $M$ and $N$ over a $\ZZ^k$-graded ring $R$, define for each $\alpha\in \ZZ^k$ the vector space 
\begin{equation}
	\uHom_R (M, N)_\alpha \coloneqq \Set{f\in \Hom_R(M,N) | f(M_\gamma) \subseteq N_{\gamma+\alpha}\text{ for all }\gamma\in \ZZ^k}
\end{equation}
of degree-$\alpha$ homomorphisms from $M$ to $N$. Define $\uHom_R(M,N)$ to be the graded $R$-module
\begin{equation}\label{eq:uHom}
	\uHom_R(M,N) \coloneqq \bigoplus_{\alpha\in \ZZ^k} \uHom_R(M,N)_\alpha,
\end{equation}
where the direct sum is taken inside $\Hom_R(M,N)$.

\subsection{Other Conventions/Notation}\label{subsec:other-convs}

\subsubsection{Varieties} Varieties (smooth or otherwise) are not required to be irreducible. A subvariety of a variety $X$ is a locally closed subset. The inclusion morphism of a subvariety $Z\subseteq X$ is usually denoted by $i_Z$, unless $Z=\{x\}$ is a point, in which case we write $i_x$ instead of $i_{\{x\}}$.

\subsubsection{Sheaves} The support of a sheaf $M$ is 
\begin{equation}\label{eq:support}
	\Supp{M}\coloneqq \Set{x\in X | M_x\neq 0}.
\end{equation}
The support of a complex $M^\bullet$ of sheaves is 
\[ \Supp M^\bullet \coloneqq \bigcup_i \Supp H^i(M^\bullet).\]

\subsubsection{Complexes and Derived Categories} If $M^\bullet$ is a (cochain) complex with differential $d_M^i\colon M^i\to M^{i+1}$ and $k\in \ZZ$, define the complex $M^\bullet[k]$ to have $i$th component
\[ M^\bullet[k]^i \coloneqq M^{k+i}\]
and differential 
\[ d_{M[k]}^i \coloneqq (-1)^k d_M^{k+i}. \]

The bounded derived category of $D_X$-modules is denoted by $\DD^b(D_X)$. The full subcategories of $\DD^b(D_X)$ generated by complexes with $D_X$-coherent and $\mathcal{O}_X$-quasicoherent cohomology are denoted by $\DD^b_c(D_X)$ and $\DD^b_{qc}(D_X)$, respectively. If $Z\subseteq X$ is a closed subvariety of $X$ and $\sharp\in\{c,qc\}$, then $\DD^{b,Z}(D_X)$ (respectively~$\DD^{b,Z}_\sharp(D_X)$) denotes the full subcategory of $\DD^b(D_X)$ (respectively~$\DD^b_\sharp(D_X)$) of complexes supported in $Z$.

\subsubsection{$D$-Modules} Given a morphism $f\colon X\to Y$ of smooth varieties, we write $f_+$ for the $D$-module direct image functor, 
\[f^+=\Lder f^*[\dim{X}-\dim{Y}]\]
for the (shifted) $D$-module inverse image functor, and
\[f_\dagger=\Ddual_Y f_+\Ddual_X\]
for the $D$-module exceptional direct image functor. 

\subsubsection{Fourier--Laplace Transform and $\widehat{\CC^n}$}\label{subsec:FL}
Recall from \eqref{eq:hatCCn} that $\widehat{\CC^n} \coloneqq \Spec R_A$ with $R_A\coloneqq \CC[\partial_1,\ldots,\partial_n]$. We identify $D_{\widehat{\CC^n}}$ with $D_{\CC^n}$ via the $\CC$-algebra isomorphism
\begin{equation}\label{eq:FLisom}
	\partial_i \mapsto \partial_i\qquad\text{and}\qquad \partial_{\partial_i}\mapsto -x_i.
\end{equation} 

The \emph{Fourier--Laplace transform} $\FL(N)$ of a $D_{\widehat{\CC^n}}$-module $N$ is $N$ viewed as a $D_{\CC^n}$-module via the isomorphism \eqref{eq:FLisom}. This functor is an exact equivalence of categories. Its inverse functor is called the \emph{inverse Fourier--Laplace transform}.

For a description of $\FL$ in terms of $D$-module direct and inverse image functors, see \cite{DE03}.

\subsection{Local Cohomology}\label{subsec:LC}
We recall the notion of local cohomology with supports in a locally closed set. As we will only need this notion for (complexes of) modules on an affine variety, we will only discuss local cohomology in this case. The reader is referred to \cite{KS90} for more detail.

Let $Z$ be a locally closed subset of an affine variety $X=\Spec R$, and let $M$ be an $R$-module. Choose an open subset $U\subseteq X$ which contains $Z$ as a closed subset. Then
\[ \Gamma_Z(M)\coloneqq \ker(\Gamma_U(M)\to \Gamma_{U\setminus Z}(M)),\]
independent of $U$. This defines a left-exact functor $\Gamma_Z$ taking $R$-modules to $R$-modules.
If $Z'$ is another locally closed subset of $X$, then $\RGamma_{Z'}\RGamma_Z\cong \RGamma_{Z'\cap Z}$. In particular, if $Z=Y\cap U$ with $Y$ closed in $X$ and $U$ open in $X$, then
\begin{equation}\label{eq:lc-cap}
	\RGamma_U\RGamma_Y\cong \RGamma_Z \cong \RGamma_Y\RGamma_U.
\end{equation}
 
Now, assume that $X$ is smooth. Then $\Gamma_Z$ takes $D_X$-modules to $D_X$-modules. The right derived functor of $\Gamma_Z\colon \Mod_{qc}(D_X)\to \Mod_{qc}(D_X)$ agrees with the derived functor of $\Gamma_Z\colon \Mod(R)\to \Mod(R)$ (here $\Mod_{qc}(D_X)$ is the category of quasi-coherent left $D_X$-modules, and $\Mod(R)$ is the category of $R$-modules).

\begin{example}\label{ex:orbitLC}
	Let $M$ be an $R_A$-module, $F\preceq A$ a face. The orbit $\orbit(F)$ is the intersection of the closed subset $V(\braket{\partial_i | \vec{a}_i\notin F})\subseteq \widehat{\CC^n}$ and the principal open subset $U=\widehat{\CC^n}\setminus V(\prod_{\vec{a}_i\in F} \partial_i)$. So, by \eqref{eq:lc-cap},
	\[\RGamma_{\orbit(F)}(M) \cong \RGamma_{\braket{\partial_i|\vec{a}_i\notin F}}(R_A[\partial^{-F}]\otimes_{R_A} M),\]
	where $\partial^{-F}\coloneqq \prod_{\vec{a}_i\in F}\partial_i^{-1}$. If $M$ is in addition a graded $R_A$-module, then $\RGamma_{\orbit(F)}(M)$ is a complex of graded $R_A$-modules.
\end{example}

\subsection{Euler--Koszul Complex}\label{subsec:Euler--Koszul} In this section, we recall the notion of Euler--Koszul complexes given in \cite{MMW05} and prove an elementary lemma (\Cref{lem:EKvsLC}) relating Euler--Koszul complexes and local cohomology. 

Define the vector
\[E_A = [E_1,\ldots,E_d]^\top\]
whose components are the Euler operators $E_i$ from \eqref{eq:euler-ops}. Given a $\ZZ^d$-graded $D_A$-module $N$ and a vector $\beta\in\CC^d$, we define an action $\circ$ of $E_i-\beta_i$ on $N$ by
\[ (E_i-\beta_i)\circ m \coloneqq (E_i -\beta_i + \deg_i(m))\cdot m\qquad (m\neq0\text{ homogeneous})\]
and extending by $\CC$-linearity. The maps $(E_i-\beta_i)\circ\colon N\to N$ are $D_A$-linear and pairwise commuting. 
\begin{definition}[{\cite[Definition~4.2]{MMW05}}]\label{def:Euler--Koszul}
	The \emph{Euler--Koszul complex} of a $\ZZ^d$-graded $R_A$-module $M$ with respect to $A$ and $\beta$ is
\[ K^A_\bullet(M; E_A-\beta) \coloneqq K_\bullet\bigl((E_A-\beta)\circ; D_A\otimes_{R_A}M\bigr);\]
i.e.~it is the Koszul complex of left $D_A$-modules defined by the sequence $(E_A-\beta)\circ$ of commuting endomorphisms on the left $D_A$-module ${D_A\otimes_{R_A}M}$. The complex is concentrated in homological degrees $d$ to $0$. The $i$th \emph{Euler--Koszul homology} is $H_i^A(M; E_A-\beta)\coloneqq H_i(K^A_\bullet(M; E_A-\beta))$. 
\end{definition}

The inverse Fourier--Laplace transform of the complex $K^A_\bullet(M; E_A-\beta)$ and the modules $H_i^A(M; E_A-\beta)$ will be denoted by $\hat{K}^A_\bullet(M; E_A-\beta)$ and $\hat{H}_i^A(M; E_A-\beta)$, respectively.

A standard computation shows that for $\alpha\in \ZZ^d$,
\begin{equation}\label{eq:ek-deg-shift}
	K^A_\bullet(M(\alpha); E_A-\beta) = K^A_\bullet(M; E-A-\beta-\alpha)(\alpha).
\end{equation}
The $\ZZ^d$-grading on the Euler--Koszul complex will usually be ignored throughout this article, so the twist by $\alpha$ on the right-hand side will usually be left out.

\begin{lemma}\label{lem:EKvsLC}
	Let $M^\bullet$ be a bounded complex of graded $R_A$-modules, and let $\beta\in \CC^d$. Then for all faces $F\preceq A$, there is a canonical isomorphism
	\[\RGamma_{\orbit(F)}\hat{K}^A_\bullet(M^\bullet; E_A-\beta) \cong \hat{K}^A_\bullet(\RGamma_{\orbit(F)}(M^\bullet); E_A-\beta).\]
\end{lemma}
\begin{proof}
	Use \Cref{ex:orbitLC} together with the fact that localization at a monomial of $R_A$ commutes with $\hat{K}^A_\bullet(-;E_A-\beta)$.
\end{proof}

\section{Fiber Support and Local Cohomology}
We now establish a relationship between fiber support, defined below, and local cohomology. The main result of this section, \Cref{prop:lc-vs-fSupp}, describes how for a sufficiently nice bounded complex $M^\bullet$ of $D$-modules (e.g.~one with holonomic cohomology), the local cohomology of $M^\bullet$ with supports in a subvariety $Z$ vanishes if and only if the fiber support of $M^\bullet$ is disjoint from $Z$. We also introduce cofiber support, which will be used later in the statement of \Cref{th:dag-plus}.

\begin{definition}\label{def:fiber-support} Let $M^\bullet\in \DD^b(\mathcal{O}_X)$.
\begin{enumerate}
	\item The \emph{fiber support} of $M^\bullet$, denoted $\fSupp{M^\bullet}$, is defined to be the set
	\[ \fSupp M^\bullet \coloneqq \Set{x\in X| k(x)\otimes^\Lder_{\mathcal{O}_{X,x}} M^\bullet_x \neq 0}.\]
	
	\item If $M^\bullet \in \DD^b_c(D_X)$, the \emph{cofiber support} of $M^\bullet$, denoted $\cofSupp{M^\bullet}$, is defined to be the set
	\[ \cofSupp{M^\bullet} \coloneqq \Set{x\in X| i_x^\dagger M^\bullet \neq 0} = \fSupp \Ddual M^\bullet.\]
\end{enumerate}	
\end{definition}

If $M^\bullet\in \DD^b(D_X)$ has regular holonomic cohomology, then its fiber support is exactly the support (recall the definition of support in \eqref{eq:support}) of the analytic solution complex $\shRHom_{D_{X^{\an}}}((M^\bullet)^{\an},\OO_{X^{\an}})$, and its cofiber support is exactly the support of the analytic de Rham complex $\Omega_{X^{\an}}\otimes_{D_{X^\an}}^\Lder (M^\bullet)^\an$, where $(-)^{\an}$ denotes analytification.

\medskip
The following two elementary lemmas are included for convenience:

\begin{lemma}\label{lem:loc-cl-restr}
	Let $X$ be a smooth variety, $Z$ a smooth subvariety, $\bar{Z}$ its closure. Then $i_Z^+$ takes $\DD^{b,\bar{Z}}_c(D_X)$ to $\DD^b_c(D_Z)$.
\end{lemma}
\begin{proof}
	Let $M^\bullet\in \DD^b_c(D_X)$. Let $U$ be an open subset of $X$ containing $Z$ in which $Z$ is closed. Then $i_U^+M^\bullet \cong i_U^{-1}M^\bullet$ is in $\DD^b_c(D_X)$ by definition of coherence. Because $i_U^+M^\bullet$ is supported on $\bar{Z}\cap U=Z$, Kashiwara's Equivalence (or more specifically \cite[Corollary~1.6.2]{htt}) then tells us that the restriction of $i_U^+M^\bullet$ to $Z$ is in $\DD^b_c(D_Z)$. This restriction is just $i_Z^+M^\bullet$.
\end{proof}

\begin{lemma}\label{lem:empty-cap}
	Let $Y,Z$ be smooth subvarieties of a smooth variety $X$, and let $i_Y,i_Z$ be their inclusions into $X$. If $Y\cap Z=\emptyset$, then $i_Y^+ i_{Z+} = 0$ and $i_Y^\dagger i_{Z\dagger}=0$ on $\DD^b_c(D_Z)$.
\end{lemma}
\begin{proof}
	Let $U=X\setminus Y$, and let $j\colon U\to X$ be inclusion. Write $i_Z'$ for the inclusion $Z\to U$. Then $i_Y^+ i_{Z+} \cong i_Y^+ j_+ i'_{Z+} = 0$,
where the isomorphism is because $i_Z = j\circ i'_Z$, and the equality is by \cite[Proposition~1.7.1(ii)]{htt}. This proves the first statement. The second statement follows by duality.
\end{proof}

\begin{proposition}\label{prop:fSupp}
	Let $X$ be a smooth variety, and let $M^\bullet\in \DD^b_c(D_X)$. Then $\fSupp M^\bullet$ is a dense subset of $\Supp M^\bullet$.
\end{proposition}
\begin{proof}
	We first show that $\fSupp M^\bullet\subseteq \Supp M^\bullet$. Let $x\in X$. If $x\notin \Supp M^\bullet$, then $M_x^\bullet \qi 0$, and therefore $k(x)\otimes^L_{\mathcal{O}_{X,x}} M^\bullet_x$ vanishes. Hence, $x\notin \fSupp M^\bullet$, proving the claim.
	
	Next, let $Y=\Supp M^\bullet$ (note that this is closed by \cite[Proposition~2.3]{milicic}). We show that the fiber support of $M^\bullet$ contains an open dense subset of $Y$; the result follows. This is accomplished in two steps: First, we show that there exists a smooth open dense subset $V\subseteq Y$ such that $i_V^+M^\bullet$ is non-zero with $\mathcal{O}_V$-projective cohomology. Second, we show that for locally projective quasi-coherent $\mathcal{O}_X$-modules, the support agrees with the fiber support.
	
	Choose a smooth dense open subset $V$ of $Y$. By \Cref{lem:loc-cl-restr}, $i_V^+M^\bullet\in \DD^b_c(D_V)$, and therefore by \cite[Proposition~3.3.2]{htt}, there exists a dense open subset $V'$ of $V$ such that all cohomology modules of $i_{V'}^+M^\bullet$ are $\mathcal{O}_{V'}$-projective. Replace $V$ with $V'$. 
	
	Suppose that $i_V^+ M^\bullet$ vanishes. Since $V$ is smooth, $\RGamma_V(M^\bullet)\cong i_{V+}i_V^+ M^\bullet$, which by assumption is zero. So, $M^\bullet \cong \RGamma_{X\setminus V}(M^\bullet)$. But $M^\bullet$ is supported in $Y$, so
	\[ \RGamma_{X\setminus V}(M^\bullet) \cong \RGamma_{Y\cap(X\setminus V)}(M^\bullet) \cong \RGamma_{Y\setminus V}(M^\bullet).\]
	Hence, $M^\bullet\cong \RGamma_{Y\setminus V}(M^\bullet)$ and therefore, using that $Y\setminus V$ is closed in $X$, is supported in $Y\setminus V$. This contradicts the fact that $V$ is dense in the non-empty set $Y=\Supp M^\bullet$. Thus, $i_V^+ M^\bullet \neq 0$, proving the first claim.
	
	To prove the second claim, let $P$ be a locally projective quasi-coherent $\mathcal{O}_X$-module. By \cite[\href{http://stacks.math.columbia.edu/tag/058Z}{Tag 058Z}]{stacks-project}, each stalk of $P$ is free, hence faithfully flat. Thus, $k(x)\otimes_{\mathcal{O}_{X,x}}^L P_x \cong k(x)\otimes_{\mathcal{O}_{X,x}} P_x$, and this vanishes if and only if $P_x$ vanishes.
\end{proof}

\begin{example} Although \Cref{prop:fSupp} tells us that the fiber support of a $D$-module is always contained in its support, this containment is in general strict: 

Consider the $D_\CC$-module $M=\mathcal{O}_\CC[x^{-1}]$, where $x$ is the coordinate function on $\CC$. The restriction of $M$ to $\CC^*$ is a (non-0) integrable connection, so the support and fiber support of $M$ both contain $\CC^*$. By \cite[Proposition~2.3]{milicic}, $\Supp M$ is closed and therefore equal to $\CC$. On the other hand, $x$ acts invertibly on the stalk $M_0$, so the (total) fiber $k(0)\otimes^\Lder_{\mathcal{O}_{\CC,0}} M_0 = 0$. Hence, $\fSupp M = \CC^*$.
\end{example}

\begin{corollary}\label{cor:fSupp}
	Let $X$ be a smooth variety, and let $M^\bullet\in \DD^b_c(D_X)$. Then $\fSupp M^\bullet$ is empty if and only if $M^\bullet\cong 0$. \qed
\end{corollary}

\begin{proposition}\label{prop:lc-vs-fSupp}
	Let $X$ be a smooth variety, $Z\subseteq X$ be a subvariety, and $M^\bullet\in \DD^b_{qc}(D_X)$. If $\RGamma_Z(M^\bullet)\cong 0$, then $Z\cap\fSupp{M^\bullet}=\emptyset$. The converse holds if both $M^\bullet$ and $\RGamma_Z(M^\bullet)$ are in $\DD^b_c(D_X)$ (e.g.~if $M^\bullet\in \DD^b_h(D_X)$).
\end{proposition}
\begin{proof}
	By Kashiwara's Equivalence, $i_x^+ \cong i_x^+ i_{x+} i_x^+$ (on $\DD^b_{qc}(D_X)$), which in turn is isomorphic to $i_x^+ \RGamma_{\{x\}}$. On the other hand, if $x\in Z$, then $\RGamma_{\{x\}} \RGamma_Z\cong \RGamma_{\{x\}}$. Combining these, we get that $i_x^+ \RGamma_Z(M^\bullet)\cong i_x^+ M^\bullet$ for all $x\in Z$. Hence, if $\RGamma_Z(M^\bullet)$ vanishes, the same applies to $i_x^+ M^\bullet$ for every $x\in Z$. This proves the first statement.
	
	To prove the second statement, let $M^\bullet \in \DD^b_c(D_X)$, and assume that ${Z\cap \fSupp{M^\bullet}=\emptyset}$. We show that $\fSupp{\RGamma_Z(M^\bullet)}=\emptyset$, so that $\RGamma_Z(M^\bullet)$ vanishes by \Cref{cor:fSupp} (note that \Cref{cor:fSupp} applies by the coherence assumption on $\RGamma_Z(M^\bullet)$). 
	
	By the first part of this proof, if $x\in Z$, then $i_x^+ \RGamma_Z(M^\bullet)\cong i_x^+ M^\bullet$, which vanishes by assumption. To see that $i_x^+ \RGamma_Z(M^\bullet)$ also vanishes for $x\notin Z$, let $U\subseteq X$ be an open neighborhood of $Z$ in which $Z$ is closed, $j\colon U\to X$ inclusion. Then 
	\begin{equation}\label{eq:x-U-Z}
		i_x^+ \RGamma_Z(M^\bullet)\cong i_x^+ \RGamma_U \RGamma_Z(M^\bullet) \cong i_x^+ j_+j^+ \RGamma_Z(M^\bullet) \cong i_x^+ j_+ \RGamma_Z(M_\bullet|_U).
	\end{equation}
	There are two cases: If $x\notin U$, then the right-hand side of \eqref{eq:x-U-Z} vanishes by \Cref{lem:empty-cap}. On the other hand, suppose $x\in U\setminus Z$. Then $i_x^+ j_+\cong (i'_x)^+ j^+ j_+ \cong (i'_x)^+$, where $i'_x\colon \{x\}\to U$ is inclusion. Combined with \eqref{eq:x-U-Z}, this gives
	\[i_x^+ \RGamma_Z(M^\bullet)\cong (i'_x)^+ \RGamma_Z(M^\bullet|_U).\] 
	But $Z$ is closed in $U$, so $\Supp \RGamma_Z(M^\bullet|_U) \subseteq Z$, which by assumption doesn't contain $x$. Hence, by \Cref{cor:fSupp}, $x\notin \fSupp \RGamma_Z(M^\bullet)$.
\end{proof}

%

\section{Quasidegrees}
In this section we prove some lemmas on quasidegrees (\Cref{def:qdeg}). These lemmas will be needed later to establish quasi-isomorphisms of certain Euler--Koszul complexes, and in \Cref{prop:EA} to establish $\mathcal{E}_A$ as the union of certain other related quasidegree sets. \Cref{lem:torsion-qdeg} provides a sufficient condition on a graded $R_A$-module $M$ for there to be a face $F\preceq A$ such that $\qdeg{M}$ is a union of translates of $\CC{F}$. \Cref{lem:dim-qdeg} states that for a finitely-generated graded $S_A$-module $M$, the quasidegree set of $M$ has the same dimension as the support of $M$.

\medskip
We begin by generalizing the definition of quasidegrees from that given in \cite[Definition~5.3]{SW09} (which is itself a generalization of \cite[Proposition~5.3]{MMW05}, where the notion originated).

\begin{definition}\label{def:qdeg}
	The \emph{true degree set} of a $\ZZ^d$-graded $R_A$-module $M$, denoted $\tdeg{M}$, is defined to be the set of $\alpha\in \ZZ^d$ such that $M_\alpha\neq0$. 
	
	The \emph{quasidegree set} of a finitely-generated $\ZZ^d$-graded $R_A$-module $M$, denoted $\qdeg{M}$, is defined to be the Zariski closure (in $\CC^d$) of $\tdeg{M}$.  We extend the definition of $\qdeg$ to arbitrary $\ZZ^d$-graded $R_A$-modules by
	\[ \qdeg{M} \coloneqq \bigcup_{M'} \qdeg M',\]
	where the union is over all finitely-generated graded submodules $M'\subseteq M$. If $M^\bullet$ is a complex of such modules, we define
	\[ \qdeg{M^\bullet}\coloneqq \bigcup_i \qdeg H^i(M^\bullet). \]
\end{definition}

Before continuing, recall from \eqref{eq:IAF} and \eqref{eq:del-to-minusF} that $I_F^A = I_A + \braket{\partial_i|\vec{a}_i\notin F}$ and $\partial^{kF}=\prod_{\vec{a}_i\in F} \partial_i^k$. Recall also the definitions of *-simple and *-free given in \S\ref{subsec:*things}.

\begin{lemma}\label{lem:torsion-qdeg}
	Let $M$ be a $\ZZ^d$-graded $R_A$-module, ${F\preceq A}$ a face. If $M$ is both an $R_A[\partial^{-F}]$-module and $I_F^A$-torsion, then every irreducible component of $\qdeg{M}$ is a translate of $\CC{F}$. Hence, 
	\[\qdeg{M} = \Set{\beta\in \CC^d| M_{\beta+\CC{F}} \neq 0}.\]
\end{lemma}
\begin{proof}
	Consider the exhaustive filtration $M_k=0:_M (I_F^A)^k$ of $M$ (this is exhaustive because $M$ is $I_F^A$-torsion). Since $M$ is an $R_A[\partial^{-F}]$-module, each $M_k$ is an $R_A[\partial^{-F}]$-submodule. Moreover, each factor module $M_k/M_{k-1}$ is by construction killed by $I_F^A$. Thus, $M_k/M_{k-1}$ is an $S_F[\partial^{-F}]$-module for all $k$. But $S_F[\partial^{-F}]$ is a *-simple ring, so each $M_k/M_{k-1}$ is a *-free $S_F[\partial^{-F}]$-module. Now, every finitely generated graded submodule of a direct sum is contained in a finite sub-sum, and the quasidegree set of a finite direct sum is the union of the quasidegree sets of its summands; so, the same is true for an infinite direct sum. Thus, $\qdeg(M_k/M_{k-1})$ is a union of translates of $\qdeg(S_F[\partial^{-F}])=\CC{F}$, proving the first claim.
	
	For the second claim, $\beta\in \qdeg{M}$ if and only if it is contained in an irreducible component of $\qdeg{M}$. But by the first claim, every irreducible component of $\qdeg{M}$ is a translate of $\CC{F}$. The only such translate containing $\beta$ is $\beta+\CC{F}$, and in the present situation this is an irreducible component of $\qdeg{M}$ if and only if it intersects $\tdeg{M}$, i.e.~if and only if $M_{\beta+\CC{F}}\neq 0$.
\end{proof}

\begin{lemma}\label{lem:dim-qdeg}
	Let $M$ be a finitely generated $\ZZ^d$-graded $S_A$-module. Then
	\[ \dim\Supp{M} = \dim\qdeg{M}.\]
\end{lemma}
\begin{proof}
	Choose a filtration $0=M_0\subseteq \cdots \subseteq M_s=M$ of $M$ by graded submodules such that each $M_i/M_{i-1}$ is of the form $S_F(-\alpha)$ for some face $F\preceq A$ and some $\alpha\in \ZZ^d$ (in the terminology of \cite[Definition~4.5]{MMW05}, $\{M_i\}$ is a \emph{toric filtration}). Then
	\[ \Supp{M} = \bigcup_{i=1}^s \Supp{M_i/M_{i-1}}.\]
	So, $\dim\Supp{M}$ is equal to the maximum of the dimensions $\dim\Supp(M_i/M_{i-1})$.
	
	On the other hand, each of the sets $\qdeg(M_i/M_{i-1})$ is a translate of the span of one of the finitely many faces of $\NN{A}$. So, the dimension of $\qdeg{M}$ is equal to the maximum of the dimensions $\dim\qdeg(M_i/M_{i-1})$.
	
	Thus, we are reduced to the case $M=S_F(-\alpha)$ for some $F\preceq A$, $\alpha\in \ZZ^d$. Then $\qdeg(S_F(-\alpha))= \CC{F}+\alpha$ and $\Supp(S_F(-\alpha))=V(I_F^A)$. Since both of these have dimension $d_F$, we arrive at the result.
\end{proof}

\begin{lemma}\label{lem:qdeg-loc}
	Let $M$ be a finitely generated graded $S_A$-module, $F\preceq G\preceq A$. Then a subset $Z\subseteq \qdeg M[\partial^{-G}]$ is an irreducible component of $\qdeg M[\partial^{-G}]$ if and only if it is an irreducible component of $\qdeg M[\partial^{-F}]$. In particular,
	\[ \qdeg{M[\partial^{-G}]}\subseteq \qdeg M[\partial^{-F}].\]
\end{lemma}
\begin{proof}
	Let $H$ be any face of $A$. Choose a filtration 
	\[ 0=M_0\subseteq \cdots \subseteq M_r = M\]
	of $M$ as in \Cref{lem:dim-qdeg}. Write $M_i/M_{i-1}\cong S_{F_i}(-\alpha_i)$. Then $\{M_i[\partial^{-F}]\}$ is a filtration of $M[\partial^{-H}]$, and its $i$th factor module is isomorphic to $S_{F_i}[\partial^{-F}](-\alpha_i)$, which is non-zero if and only if $F_i\succeq H$. Therefore,
	\begin{align*}
		\qdeg M[\partial^{-H}] 
		&= \bigcup_{i=1}^r \qdeg\bigl((M_i/M_{i-1})[\partial^{-H}]\bigr)\\
		&= \bigcup_{i=1}^r \qdeg\bigl(S_{F_i}[\partial^{-H}](-\alpha_i)\bigr)\\
		&= \bigcup_{F_i\succeq H} \qdeg\bigl(S_{F_i}[\partial^{-H}](-\alpha_i)\bigr)\\
		&= \bigcup_{F_i\succeq H} \left[\qdeg\bigl(S_{F_i}[\partial^{-H}]\bigr) + \alpha_i\right]
	\end{align*}
	If $F_i\succeq H$, then every finitely generated submodule of $S_{F_i}[\partial^{-H}]$ is contained in $S_{F_i}\partial^{-kH}$ for some $k$, and $\qdeg(S_{F_i}\partial^{-kH}) = \CC F_i + k\deg{\partial^H} = \CC F_i$. Therefore, $\qdeg S_{F_i}[\partial^{-H}] = \CC F_i$. Hence,
	\begin{equation}\label{eq:locH}
		\qdeg M[\partial^{-H}] = \bigcup_{F_i\succeq H} (\CC{F_i} + \alpha_i).
	\end{equation}
	
	Set $\mathcal{Z}_H \coloneqq \Set{ \CC F_i + \alpha_i | F_i\succeq H}$. Each $\CC F_i + \alpha_i$ is irreducible, so by \eqref{eq:locH}, the irreducible components of $\qdeg M[\partial^{-H}]$ are exactly the maximal elements of $\mathcal{Z}_H$. We show that $\mathcal{Z}_H$ is an upper subset of $\mathcal{Z}\coloneqq \mathcal{Z}_\emptyset$ (recall that a subset $Y$ of an ordered set $(X,\leq)$ is \emph{upper} if for all $y\in Y$, we have $\set{x\in X|y\leq x}\subseteq Y$). It follows that $\mathcal{Z}_G$ is an upper subset of $\mathcal{Z}_F$, and therefore that an element of $\mathcal{Z}_G$ is maximal in $\mathcal{Z}_G$ if and only if it is maximal in $\mathcal{Z}_F$, proving the lemma.
	
	Let $\CC F_i + \alpha_i \in \mathcal{Z}_H$, and suppose $\CC F_j + \alpha_j$ contains $\CC F_i + \alpha _i$. Then $F_j\succeq F_i\succeq H$, so $\CC F_j + \alpha_j\in \mathcal{Z}_H$. Thus, $\mathcal{Z}_H$ is an upper subset of $\mathcal{Z}$, as claimed.
\end{proof}

\begin{remark}
	The proofs of \Cref{lem:dim-qdeg,lem:qdeg-loc} work also for toric $R_A$-modules as defined in \cite[Definition~4.5]{MMW05}. With minor adjustments, they can even be made to work for weakly toric modules (see \cite[Definition~5.1]{SW09}).
\end{remark}

\section{The Holonomic Dual of Euler--Koszul Complexes}
The following theorem, \Cref{th:ek-dual}, will be used in \Cref{cor:phi-dagger} to give a first description of $\varphi_\dagger\mathcal{O}_{T_A}^\beta$, and then in the next section to describe $\FL(\varphi_\dagger\mathcal{O}_{T_A}^\beta)$ as an $A$-hypergeometric system.

\medskip
Before stating the theorem, we need a definition. Also recall from \eqref{eq:uHom} that $\uHom_{R_A}(M,N)$ denotes the graded $R_A$-module whose degree-$\alpha$ component is the vector space of $R_A$-module homomorphisms from $M$ to $N$ of degree $\alpha$.

\begin{definition}
	Given a bounded complex $M^\bullet$ of finitely generated graded $R_A$-modules, define
	\[ \Rdual{M^\bullet} \coloneqq \RuHom_{R_A}(M^\bullet, \omega_{R_A})[n-d],\]
	where the shift is cohomological and $\omega_{R_A}\coloneqq R_A(-\sum_i \vec{a}_i)$.
\end{definition}

\Cref{th:ek-dual} below is proved in essentially the same way as is \cite[Theorem~6.3]{MMW05}---no problems occur translating from statements about modules and spectral sequences to statements in the derived category. We therefore omit the proof of \Cref{th:ek-dual}. Note that the reason that \Cref{th:ek-dual} does not need the auto-equivalence $N\mapsto N^{-}$ as does \cite[Theorem~6.3]{MMW05} is that we work with the inverse-Fourier--Laplace-transformed Euler--Koszul complex, whereas \cite{MMW05} works with the Euler--Koszul complex itself.

Recall from \eqref{eq:epsA} that $\varepsilon_A\coloneqq \vec{a}_1+\cdots +\vec{a}_n$.

\begin{theorem}\label{th:ek-dual}
	Let $M^\bullet$ be a bounded complex of finitely-generated graded $R_A$-modules, $\beta\in \CC^d$. Then 
	\[ \Ddual \hat{K}_\bullet^A(M^\bullet; E_A - \beta) \cong \hat{K}_\bullet^A(\Rdual{M^\bullet}; E_A +\beta).\]
	If the $\ZZ^d$-grading is taken into account, the right-hand side must be twisted by $-\varepsilon_A$. \qed
\end{theorem}

\begin{definition}
	Define
	\begin{equation}
		\omega_{S_A}^\bullet \coloneqq \bigoplus_{F\preceq A}\CC\{\NN{F}-\NN{A}\},
	\end{equation}
	where the summand $\CC\{\NN{F}-\NN{A}\}$ sits in cohomological degree $d_{A/F}$, and the coboundary maps are the natural projections with signs chosen appropriately (for details, see \cite[Def.~12.7]{ms05} or \cite[\S2]{ishida}). This is a complex of *-injective modules (see \S\ref{subsubseq:*-inj}).
\end{definition}

By \cite[Theorem~3.2]{ishida}, $\omega_{S_A}^\bullet$ is a dualizing complex in the ungraded category; the arguments there show that $\omega_{S_A}^\bullet$ is also a dualizing complex in the $\ZZ^d$-graded category. With minor changes to its proof, \cite[Theorem V.3.1]{RD} implies that $\omega_{S_A}^\bullet$ is unique (in the $\ZZ^d$-graded derived category) up to cohomological shift. Its cohomological degrees are chosen such that $\uHom_{S_A}(\CC, \omega_{S_A}^\bullet)$ is quasi-isomorphic to the complex $\CC[-d]$; this choice implies that 
\begin{equation}\label{eq:DR-vs-DS}
	\Rdual(-) \cong \uHom_{S_A}(-,\omega_{S_A}^\bullet)
\end{equation}
and
\begin{equation}\label{eq:om-vs-DS}
	\omega_{S_A}^\bullet \cong \Rdual(S_A)
\end{equation}
in the derived category of graded $S_A$-modules.

\begin{remark}\label{rmk:face-LC-om}
	Let $F\preceq A$ be a face. Since $\omega_{S_A}^\bullet$ is a complex of *-injective modules, $\RGamma_{\orbit(F)}(\omega_{S_A}^\bullet)\cong \Gamma_{\orbit(F)}(\omega_{S_A}^\bullet)$. By \Cref{ex:orbitLC}, we have
	\begin{equation*}
		\Gamma_{\orbit(F)}(\omega_{S_A}^\bullet) \cong \Gamma_{\braket{\partial_i|\vec{a}_i\notin F}}(R_A[\partial^{-F}]\otimes_{R_A} \omega_{S_A}^\bullet).
	\end{equation*}
	If a face $F'\preceq A$ does not contain $F$, then $\CC\{\NN{F}-\NN{A}\}$ is $\partial^F$-torsion and therefore vanishes upon tensoring with $R_A[\partial^{-F}]\otimes_{R_A}$. Then because $\braket{\partial_i|\vec{a}_i\notin F}S_A$ is the homogeneous prime ideal corresponding to $F$, the only module $\CC\{\NN{F'}-\NN{A}\}$ with $F'\succeq F$ which is not killed by $\Gamma_{\braket{\partial_i|\vec{a}_i\notin F}}$ is $\CC\{\NN{F}-\NN{A}\}$. Hence,
	\begin{equation}\label{eq:face-LC-om}
		\RGamma_{\orbit(F)}(\omega_{S_A}^\bullet) \cong \CC\{\NN{F}-\NN{A}\}[-d_{A/F}].
	\end{equation}
\end{remark}

Recall that $\sRes(A)$ was defined in \eqref{eq:sres}.

\begin{corollary}\label{cor:phi-dagger}
	If $-\beta\in \CC^d\setminus \sRes(A)$, then $\varphi_\dagger \mathcal{O}_{T_A}^\beta \cong \hat{K}_\bullet^A(\omega_{S_A}^\bullet; E_A-\beta)$.
\end{corollary}
\begin{proof}
	The holonomic dual of $\mathcal{O}_{T_A}^\beta$ is $\mathcal{O}_{T_A}^{-\beta}$, and by \cite[Corollary~3.7]{SW09}, applying $\varphi_+$ to this gives $\hat{K}^A_\bullet(S_A; E_A+\beta)$. So, by \Cref{th:ek-dual}, 
	\[\varphi_\dagger \mathcal{O}_{T_A}^\beta\cong \Ddual \varphi_+\Ddual \mathcal{O}_{T_A}^\beta\cong  \hat{K}^A_\bullet(\Rdual{S_A};E_A-\beta).\]
	Now use \eqref{eq:om-vs-DS}.
\end{proof}

\section{The Exceptional Direct Image of $\mathcal{O}_{T_A}^\beta$}
Reichelt proves in \cite[Proposition~1.14]{Rei14} that $\FL(\varphi_\dagger\mathcal{O}_{T_A}^\beta)$ is isomorphic to a GKZ system for homogeneous $A$ and $\beta\in \QQ^d$. We now generalize this to arbitrary $A$, $\beta$. This generalization, or rather \Cref{prop:phi-dagger}, will be used later in the proof of \Cref{th:plus-dag}.

\begin{lemma}\label{lem:omega-qdeg}
	For all $i\in \NN$, $\dim\qdeg H^i(\omega_{S_A}^\bullet)\leq d-i$.
\end{lemma}
\begin{proof}
	By definition, $\omega_{S_A}^i$ is the direct sum of $\CC\{\NN{F}-\NN{A}\}$ for faces $F$ with $d_{A/F}=i$. Each $\CC\{\NN{F}-\NN{A}\}$ has support equal to $V(I_F^A)$, which has dimension $d_F=d-i$. So, $\dim\Supp\omega_{S_A}^i = d-i$. Hence, because $H^i(\omega_{S_A}^\bullet)$ is a subquotient of $\omega_{S_A}^i$, its support must have dimension at most $d-i$. Now apply \Cref{lem:dim-qdeg}.
\end{proof}

\begin{proposition}\label{prop:phi-dagger}
	Let $\beta\in \CC^d$. Then for all ${k\gg0}$,
	\[\varphi_\dagger\mathcal{O}_{T_A}^\beta \cong \hat{K}^A_\bullet(S_A; E_A-\beta + k\varepsilon_A).\]
\end{proposition}
\begin{proof}
	First, notice that by \cite[Corollaries~3.1 and 3.7]{SW09}, $-\beta+k\varepsilon_A\notin \sRes(A)$ for all $k\gg0$. Also notice that $\mathcal{O}_{T_A}^\beta\cong \mathcal{O}_{T_A}^{\beta'}$ for all $\beta'\equiv \beta \pmod{\ZZ^d}$. Hence, in light of \Cref{cor:phi-dagger}, we may replace $\beta$ with $\beta-k\varepsilon_A$ to assume that $-\beta\notin\sRes(A)$.
	
	\underline{Step 1:} 
	We show that $\beta-k\varepsilon_A\notin \qdeg\cone(H^0(\omega_{S_A}^\bullet)\to \omega_{S_A}^\bullet)$ for all ${k\gg0}$. Then, applying \cite[Theorem 5.4(3)]{SW09} along with a basic spectral sequence argument, it follows that $\hat{K}^A_\bullet(-;E_A-\beta+k\varepsilon_A)$ applied to the morphism $H^0(\omega_{S_A}^\bullet)\to \omega_{S_A}^\bullet$ is a quasi-isomorphism for all $k\gg0$.
	
	By \Cref{lem:omega-qdeg}, the union $\bigcup_{i>0} \qdeg H^i(\omega_{S_A}^\bullet)$ has codimension at least 1, and because each cohomology module of $\omega_{S_A}^\bullet$ is finitely generated, this union has finitely many irreducible components. Therefore, since $\varepsilon_A$ is in the relative interior of $\NN{A}$, we see that $\beta -k\varepsilon_A \notin \bigcup_{i>0} \qdeg H^i(\omega_{S_A}^\bullet)$ for all $k\gg0$. But the $i$th cohomology of $\cone(H^0(\omega_{S_A}^\bullet)\to \omega_{S_A}^\bullet)$ is $0$ if $i\leq 0$ and is $H^i(\omega_{S_A}^\bullet)$ if $i>0$. Hence, $\beta-k\varepsilon_A\notin\cone(H^0(\omega_{S_A}^\bullet)\to \omega_{S_A}^\bullet)$ for all ${k\gg0}$, as promised.
	
	\underline{Step 2:}
	We construct a quasi-isomorphism
	\[\hat{K}^A_\bullet(S_A;E_A-\beta + k\varepsilon_A) \xrightarrow{\;\mathrm{qi}\;} \hat{K}^A_\bullet(H^0(\omega_{S_A}^\bullet);E_A-\beta+ k\varepsilon_A)\]
	for ${k\gg0}$. Let $0\neq m\in H^0(\omega_{S_A}^\bullet)$ be homogeneous. Since $H^0(\omega_{S_A}^\bullet)\subseteq \CC[\ZZ^d]$ and is non-zero (it contains $m$), the zero ideal is one of its associated primes (in fact the only one). Therefore, $H^0(\omega_{S_A}^\bullet)$ must contain a twist of $S_A$; in particular, it must contain $k_0\varepsilon_A$ for some $k_0\in \NN$. Hence, $m$ may be chosen to have degree $k_0\varepsilon_A$.
	
	Now, consider the quotient $H^0(\omega_{S_A}^\bullet)/S_Am$. The quasidegree set of this quotient has codimension at least 1, so as before, $\beta -k\varepsilon_A \notin \qdeg(H^0(\omega_{S_A}^\bullet)/S_Am)$ for all $k\gg0$. Hence, the morphism
	\[ \hat{K}^A_\bullet(S_A(-k_0\varepsilon_A);E_A-\beta + k\varepsilon_A) \to \hat{K}^A_\bullet(H^0(\omega_{S_A}^\bullet); E_A - \beta + k\varepsilon_A)\]
	induced by right-multiplication by $m$ is a quasi-isomorphism for $k\gg0$. Applying \eqref{eq:ek-deg-shift} gives the result.
\end{proof}

The promised generalization is given by the following corollary:

\begin{corollary}
	Let $\beta\in \CC^d$. Then for all $k\gg0$,
	\[ \FL(\varphi_\dagger\mathcal{O}_{T_A}^\beta) \cong M_A(\beta-k\varepsilon_A).\]
\end{corollary}
\begin{proof}
	By \Cref{prop:phi-dagger}, it suffices to show that $\varphi_\dagger\mathcal{O}_{T_A}^\beta$ has cohomology only in degree $0$. The holonomic dual of $\varphi_\dagger\mathcal{O}_{T_A}^\beta$ is $\varphi_+\mathcal{O}_{T_A}^{-\beta}$, which by \cite[Proposition~2.1]{SW09} has cohomology only in degree $0$. Then because $\Ddual$ is exact, the same applies to $\varphi_\dagger\mathcal{O}_{T_A}^\beta$.
\end{proof}

\section{Restricting Euler--Koszul Complexes to Orbits}
We now compute the restriction and exceptional restriction to an orbit (as defined in \eqref{eq:orbit}) of an (inverse Fourier-transformed) Euler--Koszul complex in terms of local cohomology and localizations, respectively. Recall from \S\ref{subsec:*things} that a *-injective $S_A$-module is an injective object in the category of $\ZZ^d$-graded $S_A$-modules, and every indecomposable *-injective $S_A$-module is a $\ZZ^d$-graded twist of $\CC\{\NN{F}-\NN{A}\}$ for some face $F\preceq A$.

\begin{proposition}\label{prop:injectives}
	Let $J^\bullet$ be a bounded below complex of *-injective $S_A$-modules, $\beta\in \CC^d$, and $F\preceq A$. Assume that each $J^i$ is a direct sum of twists of $\CC\{\NN{F}-\NN{A}\}$. Then there exists a quasi-isomorphism of double complexes\footnote{By a \emph{quasi-isomorphism of double complexes} between $M^{\bullet\bullet}$ and $M'^{\bullet\bullet}$, we mean a pair of morphisms 
	\[ M^{\bullet\bullet} \xleftarrow{\;f\;} N^{\bullet\bullet} \xrightarrow{\;g\;} M'^{\bullet\bullet}\] such that $\Tot(f)$ and $\Tot(g)$ are quasi-isomorphisms of complexes.}
\[ \{K_{-p}(J^q;E_A-\beta)\} \qi \left\{\bigoplus_{\lambda+\ZZ{F} \in \CC{F}/\ZZ{F}} K_{-p}(S_F[\partial^{-F}]; E_A-\lambda)\otimes_\CC J^q_{\beta-\lambda}\right\}.\]
\end{proposition}
\begin{proof} Consider the subcomplexes of $J^\bullet$ given by $M^\bullet = S_A[\partial^{-F}] J^\bullet_{\beta+\CC{F}}$ and $M'^\bullet = \Braket{\partial_i|\vec{a}_i\notin F} M^\bullet$ (note that $J^\bullet$ is a complex of $S_A[\partial^{-F}]$-modules, so $M^\bullet$ is in fact a subcomplex of $J^\bullet$). We claim for all $q$ that $\beta\notin \qdeg(M'^q)$ and $\beta\notin \qdeg(J^q/M^q)$. To see this, notice that the intersections of $\beta + \CC{F}$ with $\tdeg(M'^q)$ and with $\tdeg(J^q/M^q)$ are both empty by construction. But both $M'^q$ and $J^q/M^q$ are $I_F^A$-torsion (because $J^q$ is). Hence, by \Cref{lem:torsion-qdeg}, $\beta$ is a quasidegree of neither, proving the claim.
	
	From the claim, we get that for all $q$, the morphisms 
	\[ K_\bullet(M^q; E_A-\beta) \to K_\bullet(J^q; E_A-\beta)\]
and
	\[ K_\bullet(M^q; E_A-\beta) \to K_\bullet(M^q/M'^q; E_A-\beta)\]
	are both quasi-isomorphisms, and therefore we get a quasi-isomorphism of double complexes
	\begin{equation}\label{eq:M-mod-M'}
		\{K_{-p}(J^q;E_A-\beta)\} \qi \{K_{-p}(M^q/M'^q; E_A-\beta)\}.
	\end{equation}
	
	Next, notice that $M^\bullet/M'^\bullet$ is a complex of graded modules over $\CC[\ZZ{F}]$ (which we identify with $S_F[\partial^{-F}]$). So, since $\CC[\ZZ{F}]$ is a *-simple ring, each $M^q/M'^q$ is a direct sum of $\ZZ^d$-graded twists of $\CC[\ZZ{F}]$. Therefore, by gradedness,
	\begin{equation}\label{eq:alpha-split}
		M^\bullet/M'^\bullet = \bigoplus_{\alpha +\ZZ{F} \in \ZZ^d/\ZZ{F}} (M^\bullet/M'^\bullet)_\alpha\otimes_\CC \CC[\ZZ{F}](-\alpha).
	\end{equation}
	Now, as complexes of vector spaces, $(M^\bullet/M'^\bullet)_\alpha$ is isomorphic to $J^\bullet_\alpha$ if $\alpha\in \beta + \CC F$ and is zero otherwise. So, combining this with eqs.~\eqref{eq:M-mod-M'} and \eqref{eq:alpha-split}, we get
	\begin{align}
		\{K_{-p}(J^q;E_A-\beta)\} 
		&\qi \left\{\bigoplus_{\substack{\alpha +\ZZ{F} \in \ZZ^d/\ZZ{F}\\ \alpha\in \beta+\CC F}} K_{-p}(\CC[\ZZ{F}](-\alpha)\otimes_\CC J^q_\alpha; E_A-\beta)\right\} \notag\\
		&= \left\{\bigoplus_{\substack{\alpha +\ZZ{F} \in \ZZ^d/\ZZ{F}\\ \alpha\in \beta+\CC F}} K_{-p}(\CC[\ZZ{F}](-\alpha); E_A-\beta)\otimes_\CC J^q_\alpha\right\}.
	\end{align}
	But $K_\bullet(\CC[\ZZ{F}](-\alpha); E_A-\beta) = K_\bullet(\CC[\ZZ{F}]; E_A - \beta + \alpha)$. So, re-indexing the sum, we are done.
\end{proof}


Let $M^\bullet$ be a bounded complex of $\ZZ^d$-graded $S_A$-modules, let $\beta\in \CC^d$, and let $F\preceq A$ be a face. For $\lambda\in \CC{F}$, we give $\RGamma_{\orbit(F)}(M^\bullet)_{\beta-\lambda+\ZZ{F}}$ the structure of a complex of $D_{T_F}$-modules as follows: Let $m$ be a homogeneous element of $H^i_{\orbit(F)}(S_A)_{\beta-\lambda + \ZZ{F}}$ for some $i$. Recalling the definition of $\vartheta_u$ from \eqref{eq:vartheta}, we set 
\begin{equation}\label{eq:DTF-struct}
	\vartheta_u\cdot m \coloneqq \braket{\deg(m)-\beta, u}m \qquad (u\in (\CC{F})^*).
\end{equation}
Observing that \eqref{eq:DTF-struct} makes no reference to $\lambda$, we get an isomorphism
\begin{equation}\label{eq:DTF-struct-lc}
	\bigoplus_{\lambda+\ZZ{F}\in \CC{F}/\ZZ{F}} \RGamma_{\orbit(F)}(M^\bullet)_{\beta-\lambda+\ZZ{F}} \cong \RGamma_{\orbit(F)}(M^\bullet)_{\beta+\CC{F}}.
\end{equation}

In the theorem below, we use the convention that $\Bigwedge\CC^k$ lives in cohomological degrees $-k$ through $0$. 

\begin{theorem}\label{th:orbit-restr}
	Let $M^\bullet$ be a bounded complex of $\ZZ^d$-graded $S_A$-modules, $\beta\in \CC^d$. Then for all faces $F\preceq A$,
	\[ i^+_{\orbit(F)}\hat{K}_\bullet(M^\bullet; E_A-\beta) \cong \bigoplus_{\lambda+\ZZ{F}\in \CC{F}/\ZZ{F}} \mathcal{O}_{T_F}^\lambda \otimes_\CC \RGamma_{\orbit(F)}(M^\bullet)_{\beta-\lambda} \otimes_\CC \Bigwedge\CC^{d_{A/F}}.\]
	This isomorphism is functorial in $M^\bullet$.
	
	An equivalent presentatjon, absorbing the $\mathcal{O}_{T_F}^\lambda$ into the local cohomology, is
	\[ i^+_{\orbit(F)}\hat{K}_\bullet(M^\bullet; E_A-\beta) \cong \bigoplus_{\lambda+\ZZ{F}\in \CC{F}/\ZZ{F}} \RGamma_{\orbit(F)}(M^\bullet)_{\beta-\lambda+\ZZ{F}} \otimes_\CC \Bigwedge\CC^{d_{A/F}}.\]
	This can be further compacted using \eqref{eq:DTF-struct-lc} to give
	\begin{equation}\label{eq:orbit-restr-cpt}
		i^+_{\orbit(F)}\hat{K}_\bullet(M^\bullet; E_A-\beta) \cong  \RGamma_{\orbit(F)}(M^\bullet)_{\beta+\CC{F}}\otimes_\CC \Bigwedge\CC^{d_{A/F}}.
	\end{equation}
\end{theorem}
\begin{proof}
	Let $J^\bullet$ be a (bounded below) *-injective $S_A$-module resolution of $M^\bullet$. Then $\RGamma_{\orbit(F)}(M^\bullet)\cong \Gamma_{\orbit(F)}(J^\bullet)$, which is a complex of *-injective $S_A$-modules each of which is either $0$ or has $I_F^A$ as its only associated prime; that is, each $\Gamma_{\orbit(F)}(J^i)$ is a direct sum of twists of $\CC\{\NN{F}-\NN{A}\}$. Thus, noting that $i^+_{\orbit(F)} \cong i^+_{\orbit(F)}\RGamma_{\orbit(F)}$, \Cref{prop:injectives} and \Cref{lem:EKvsLC} give
	\begin{align}
		i^+_{\orbit(F)}\hat{K}_\bullet(M^\bullet; E_A-\beta)
		&\cong i^+_{\orbit(F)}\hat{K}_\bullet(\RGamma_{\orbit(F)}(M^\bullet); E_A-\beta)\notag\\
		&\cong \bigoplus_{\lambda+\ZZ{F}\in \CC{F}/\ZZ{F}} i^+_{\orbit(F)}\hat{K}_\bullet(S_F[\partial^{-F}]; E_A-\lambda)\otimes_\CC \RGamma_{\orbit(F)}(M^\bullet)_{\beta-\lambda}.\label{eq:orbit-restr-dirsum}
	\end{align}
	But for $\lambda\in \CC{F}$, 
	\begin{equation}\label{eq:wedges}
		i^+_{\orbit(F)}\hat{K}_\bullet(S_F[\partial^{-F}]; E_A-\lambda) 
		\cong i^+_{\orbit(F)}\hat{K}^F_\bullet(S_F[\partial^{-F}]; E_F-\lambda)\otimes_\CC \Bigwedge\CC^{d_{A/F}}.
	\end{equation}
	Now by \cite[Prop.~2.1]{SW09}, $\hat{K}^F_\bullet(S_F[\partial^{-F}]; E_F-\lambda)$ is isomorphic to the direct image $\varphi_{F+}\mathcal{O}_{T_F}^\lambda$, where $\varphi_F$ is the torus embedding of $T_F$ into $V(\partial_i\mid\vec{a}_i\notin F)$. Then because $i_{\orbit(F)}^+\varphi_{F+}\cong \operatorname{id}$, we get that
	\begin{equation*}
		i^+_{\orbit(F)}\hat{K}_\bullet(S_F[\partial^{-F}]; E_A-\lambda) \cong \mathcal{O}_{T_F}^\lambda.
	\end{equation*}
	Combining this with \eqref{eq:orbit-restr-dirsum} and \eqref{eq:wedges} gives the result.
\end{proof}

Before stating \Cref{th:exc-restr}, we recall the notion of ($\ZZ^d$-graded) Matlis duality:

\begin{definition}\label{def:matlis}
	Let $Q$ be an affine semigroup. The \emph{Matlis dual} of the graded $\CC[Q]$-module $M$ is the graded $\CC[Q]$-module $M^\vee \coloneqq \uHom_\CC(M,\CC)$.
\end{definition}

\begin{theorem}\label{th:exc-restr}
	Let $\beta\in \CC^d$, and let $M$ be a finitely generated graded $S_A$-module. Then for all faces $F\preceq A$,
	\[ i_{\orbit(F)}^\dagger \hat{K}^A_\bullet(M; E_A-\beta) \cong \bigoplus_{\lambda+\ZZ{F}\in\CC{F}/\ZZ{F}} \mathcal{O}_{T_F}^\lambda\otimes_\CC M[\partial^{-F}]_{\beta-\lambda} \otimes_\CC\Bigwedge\CC^{d_{A/F}}. \]
	This isomorphism is functorial in $M$.
	
	As in \Cref{th:orbit-restr}, this isomorphism may also be written as
	\[ i_{\orbit(F)}^\dagger \hat{K}^A_\bullet(M; E_A-\beta) \cong \bigoplus_{\lambda+\ZZ{F}\in\CC{F}/\ZZ{F}} M[\partial^{-F}]_{\beta-\lambda+\ZZ{F}} \otimes_\CC\Bigwedge\CC^{d_{A/F}} \]
	and as
	\[ i_{\orbit(F)}^\dagger \hat{K}^A_\bullet(M; E_A-\beta) \cong M[\partial^{-F}]_{\beta+\CC{F}} \otimes_\CC\Bigwedge\CC^{d_{A/F}}.\]
\end{theorem}
\begin{proof}
	By \Cref{th:ek-dual,th:orbit-restr},
	\begin{equation*}
		i_{\orbit(F)}^+ \Ddual\hat{K}^A_\bullet(M; E_A-\beta) \cong \bigoplus_{\lambda+\ZZ{F}\in\CC{F}/\ZZ{F}} \mathcal{O}_{T_F}^\lambda\otimes_\CC \RGamma_{\orbit(F)}(\Rdual{M})_{-\beta-\lambda} \otimes_\CC\Bigwedge\CC^{d_{A/F}}.
	\end{equation*}
	Dualizing, we get
	\begin{align*}
		i_{\orbit(F)}^\dagger\hat{K}^A_\bullet(M; E_A-\beta) 
			&\cong \bigoplus_{\lambda+\ZZ{F}\in\CC{F}/\ZZ{F}} \mathcal{O}_{T_F}^{-\lambda}\otimes_\CC [\RGamma_{\orbit(F)}(\Rdual{M})_{-\beta-\lambda}]^* \otimes_\CC\Bigwedge\CC^{d_{A/F}}[-d_{A/F}]\\
			&\cong \bigoplus_{\lambda+\ZZ{F}\in\CC{F}/\ZZ{F}} \mathcal{O}_{T_F}^{-\lambda}\otimes_\CC [\RGamma_{\orbit(F)}(\Rdual{M})^\vee]_{\beta+\lambda} \otimes_\CC\Bigwedge\CC^{d_{A/F}}[-d_{A/F}]\\
			&\cong \bigoplus_{\lambda+\ZZ{F}\in\CC{F}/\ZZ{F}} \mathcal{O}_{T_F}^\lambda\otimes_\CC [\RGamma_{\orbit(F)}(\Rdual{M})^\vee]_{\beta-\lambda} \otimes_\CC\Bigwedge\CC^{d_{A/F}}[-d_{A/F}],
	\end{align*}
	where $(-)^*$ is the vector space duality functor. In the notation of \Cref{ex:orbitLC}, we have
	\[\RGamma_{\orbit(F)}(\Rdual{M}) \cong \RGamma_{\braket{\partial_i|\vec{a}_i\notin F}}\left(\Rdual(M)[\partial^{-F}]\right).\]
	So, by \eqref{eq:DR-vs-DS} and because $M$ is finitely generated, we get
	\begin{align*}
		\RGamma_{\orbit(F)}(\Rdual{M}) 
		&\cong \RGamma_{\braket{\partial_i|\vec{a}_i\notin F}}(S_A) \otimes^\Lder_{S_A} \RuHom_{S_A}(M,\omega_{S_A}^\bullet)[\partial^{-F}]\\
		&\cong \RGamma_{\braket{\partial_i|\vec{a}_i\notin F}}(S_A) \otimes^\Lder_{S_A} \RuHom_{S_A}\!\left(M, \omega_{S_A}^\bullet[\partial^{-F}]\right)\\
		&\cong \RuHom_{S_A[\partial^{-F}]}\!\left(M[\partial^{-F}], \RGamma_{\braket{\partial_i|\vec{a}_i\notin F}}(S_A) \otimes^\Lder_{S_A} \omega_{S_A}^\bullet[\partial^{-F}]\right)\\
		&\cong \RuHom_{S_A[\partial^{-F}]}\!\left(M[\partial^{-F}], \RGamma_{\braket{\partial_i|\vec{a}_i\notin F}}\!\left( \omega_{S_A}^\bullet[\partial^{-F}]\right)\right)\\
		&\cong \RuHom_{S_A[\partial^{-F}]}\!\left(M[\partial^{-F}], \RGamma_{\orbit(F)}(\omega_{S_A}^\bullet)\right)\\
		&\cong \uHom_{S_A[\partial^{-F}]}\!\left(M[\partial^{-F}], \Gamma_{\orbit(F)}(\omega_{S_A}^\bullet)\right).
	\end{align*}
	But $\Gamma_{\orbit(F)}(\omega_{S_A}^\bullet) \cong \CC\{\NN{F}-\NN{A}\}[-d_{A/F}]$ by \eqref{eq:face-LC-om}, and $\CC\{\NN{F}-\NN{A}\}\cong S_A[\partial^{-F}]^\vee$. So, applying \cite[Lem.~11.16]{ms05}, we see that
	\[ \RGamma_{\orbit(F)}(\Rdual M)^\vee \cong (M[\partial^{-F}]^\vee[-d_{A/F}])^\vee \cong M[\partial^{-F}][d_{A/F}].\]
	Therefore,
	\begin{align*}
		i_{\orbit(F)}^\dagger\hat{K}^A_\bullet(M; E_A-\beta)
		&\cong \bigoplus_{\lambda+\ZZ{F}\in\CC{F}/\ZZ{F}} \mathcal{O}_{T_F}^\lambda\otimes_\CC M[\partial^{-F}]_{\beta-\lambda} \otimes_\CC\Bigwedge\CC^{d_{A/F}},
	\end{align*}
	as hoped.
\end{proof}


\section{$A$-Hypergeometric Systems via Direct Images}\label{sec:dirims}
In subsection~\ref{subsec:exc}, we introduce the notion of strongly $(A,F)$-exceptional quasidegrees and prove some related lemmas. In subsection~\ref{subsec:contiguity}, we study an effect of contiguity on Euler--Koszul complexes. We then state and prove the main theorems, \Cref{th:plus-dag,th:dag-plus}, in subsection~\ref{subsec:main-th}.

\medskip
Given an open subset $U\subseteq \widehat{\CC^n}$ containing (the image of) $T_A$, consider the inclusion maps in the commutative diagram below:
\[\begin{tikzcd}
		& (\CC^*)^n \arrow[dr, hook, "\varpi"] \arrow[dd, hook, "j_U"]	&\\
	T_A \arrow[ur, hook, "\iota"] \arrow[dr, hook, "\iota_U"]	&	& \widehat{\CC^n}\\
		& U \arrow[ur, hook, "\varpi_U"]	&
\end{tikzcd}\]
The morphisms $\iota$ and $\iota_U$ are the torus embedding $\varphi$ with codomain restricted to $(\CC^*)^n$ and $U$, respectively. The remaining morphisms are the inclusions.

\begin{lemma}\label{lem:dag-plus}
	With notation as above, there are, for every $M^\bullet\in\DD^b_c(D_{T_A})$, natural isomorphisms $\varpi_{U+}\iota_{U\dagger}M^\bullet\cong \RGamma_U\varphi_\dagger M^\bullet$ and $\varpi_{U\dagger}\iota_{U+}M^\bullet\cong \varpi_{U\dagger}\varpi_U^{-1}\varphi_+M^\bullet$.
\end{lemma}
\begin{proof}
	The map $j_U$ is an affine open immersion, so $j_{U+}\cong \Rder j_{U*}\cong j_{U*}$. Now, for any open subset $V\subseteq U$ and any sheaf $F$ on $(\CC^*)^n$, one has 
	\[ \Gamma(V, \varpi_U^*\varpi_* F) = \Gamma(V, \varpi_* F) = \Gamma(V\cap(\CC^*)^n, F) = \Gamma(V, j_{U*}F),\]
	so $j_{U*}=\varpi_U^*\varpi_*$, which is isomorphic to $\varpi_U^+\varpi_+$ because $\varpi_U$ is an open immersion and $\varpi$ is an affine open immersion. So, $j_{U+}\cong \varpi_U^+\varpi_+$. Therefore,
	\[ \iota_{U\dagger} 
			= \Ddual_U \iota_{U+} \Ddual_{T_A} 
			\cong \Ddual_U j_{U+}\iota_+\Ddual_{T_A}
			\cong \Ddual_U \varpi_U^+\varpi_+ \iota_+\Ddual_{T_A}
			\cong \varpi_U^+\Ddual_{\widehat{\CC^n}}\varpi_+\iota_+\Ddual_{T_A}.\]
	Since $\Ddual_{\widehat{\CC^n}}\varpi_+\iota_+\Ddual_{T_A}\cong \varphi_\dagger$ and $\varpi_{U+}\varpi_U^+\cong \RGamma_U$, we get the first isomorphism. The second isomorphism follows via duality.
\end{proof}

\subsection{Exceptional and Strongly Exceptional Quasidegrees}\label{subsec:exc}
In this section we introduce the notion of strongly $(A,F)$-exceptional quasidegrees for $F\preceq A$. These are then related in \Cref{prop:EA} to the set 
\[\mathcal{E}_A\coloneqq \bigcup_{i>0} \qdeg H^i(\omega_{S_A}^\bullet)\]
of $A$-exceptional quasidegrees. In \Cref{lem:open-fsupp}, we prove that $\hat{K}^A_\bullet(S_A;E_A-\beta)$ has relatively open fiber support if $\beta\notin \mathcal{E}_A$.

\begin{definition}\label{def:exc-qdeg}
	Given a face $F$, we define the set of \emph{strongly $(A,F)$-exceptional quasidegrees} to be
	\[ \mathcal{E}_{A,F}^{\mathrm{strong}}\coloneqq \bigcup_{i<d_{A/F}} \qdeg \HH^i_{\orbit(F)}(S_A).\]
	When $F=\emptyset$, this is just the set of strongly $A$-exceptional quasidegrees defined in \cite[Definition~2.9]{rw17}. More generally, if $M$ is a graded $R_A$-modules, we define the set of \emph{strongly $(A,F)$-exceptional quasidegrees for $M$} to be
	\[ \mathcal{E}_{A,F}^{\mathrm{strong}}(M)\coloneqq \bigcup_{i<d_{A/F}} \qdeg \HH_{\orbit(F)}^i(M).\]
\end{definition}

\begin{remark}\label{rmk:strong-CM}
	From \Cref{ex:orbitLC}, we know that
	\[ H^i_{\orbit(F)}(S_A) \cong H^i_{\braket{\partial_i|\vec{a}_i\notin F}}(S_A[\partial^{-F}]).\]
	The ideal $\braket{\partial_i|\vec{a}_i\notin F}S_A[\partial^{-F}]$ is the maximal homogeneous ideal of $S_A[\partial^{-F}]$, so $\mathcal{E}_{A,F}^{\mathrm{strong}}=\emptyset$ if and only if the affine semigroup ring $S_A[\partial^{-F}]$ is Cohen--Macaulay.
\end{remark}

\begin{example}
	If $d\leq 2$, then the localization $S_A[\partial^{-F}]$ is Cohen--Macaulay for all faces $F\neq \emptyset$. Therefore, $\mathcal{E}_{A,F}^{\mathrm{strong}} = \emptyset$ for all faces $F\neq \emptyset$.
\end{example}
\begin{example}
	Let
	\[ A = \begin{bmatrix}
		1 &0 &1 &0 &0\\
		0 &1 &0 &1 &0\\
		0 &0 &1 &1 &2
	\end{bmatrix}.\]
	The semigroup $\NN{A}$ is equal to $\NN^3\setminus\Set{(0,0,c)|c \text{ is odd}}$. Let $F\preceq A$ be a face. If $F=A$ or $\CC{F}$ does not contain the $z$-axis, then the semigroup ring $S_A[\partial^{-F}]$ is normal, hence Cohen--Macaulay, and therefore $\mathcal{E}_{A,F}^{\mathrm{strong}}=\emptyset$ by \Cref{rmk:strong-CM}. If $\CC{F}$ equals the $z$-axis, then $H^0_{\orbit(F)}(S_A)$ is zero, so $\mathcal{E}_{A,F}^{\mathrm{strong}} = \qdeg H^1_{\orbit(F)}(S_A)=\CC{F}$. If $F=\emptyset$, then $H^i_{\orbit(F)}(S_A)$ is zero if $i=0$ or $1$, so $\mathcal{E}_{A,F}^{\mathrm{strong}} = \qdeg H^2_{\orbit(F)}(S_A)=\Set{(0,0,c)|c\in \ZZ_{<0}\text{ and }c\text{ is odd}}$.
\end{example}

\begin{proposition}\label{prop:EA}
	$\displaystyle \mathcal{E}_A = \bigcup_{F\preceq A} \mathcal{E}_{A,F}^{\mathrm{strong}}$.
\end{proposition}
\begin{proof}
	It suffices to show that
	\begin{equation*}
		\qdeg H^{>0}(\omega_{S_A}^\bullet) = \bigcup_{F\preceq A} \Set{\beta\in \CC^d| H^{>0}(\omega_{S_A[\partial^{-F}]}^\bullet)_{\beta+\CC{F}}\neq0}.
	\end{equation*}
	
	\noindent $\bm{(\subseteq)}$ Let $Z=\CC{F}+\beta$ be an irreducible component of $\qdeg H^{>0}(\omega_{S_A}^\bullet)$. Then by \Cref{lem:qdeg-loc}, $Z$ is also an irreducible component of $\qdeg H^{>0}(\omega_{S_A}^\bullet)[\partial^{-F}]$, and therefore $H^{>0}(\omega_{S_A}^\bullet)[\partial^{-F}]_{\beta+\CC{F}}\neq 0$. Now use that
	\begin{equation}\label{eq:dual-loc}
		H^{>0}(\omega_{S_A}^\bullet)[\partial^{-F}]\cong H^{>0}(\omega_{S_A[\partial^{-F}]}^\bullet).
	\end{equation}
	
	\noindent $\bm{(\supseteq)}$ Suppose $H^{>0}(\omega_{S_A[\partial^{-F}]}^\bullet)_{\beta+\CC{F}}\neq 0$. Then by \eqref{eq:dual-loc} and \Cref{lem:qdeg-loc}, the irreducible component of $\qdeg H^{>0}(\omega_{S_A[\partial^{-F}]}^\bullet)$ containing $\beta$ is also an irreducible component of $\qdeg H^{>0}(\omega_{S_A}^\bullet)$.
\end{proof}

The proof of \Cref{lem:open-fsupp} requires the Ishida complex of an affine semigroup ring, which we now recall.

\begin{definition}\label{def:ishida}
	Let $Q$ be an affine semigroup, $\pi\coloneqq \RR_{\geq 0} Q$ its cone. The \emph{Ishida complex} of $\CC[Q]$ is the complex
	\begin{equation}
		\mho^\bullet_{\CC[Q]} \coloneqq \bigoplus_{\sigma\text{ a face of }\pi} \CC[Q]_\sigma,
	\end{equation}
	where $\CC[Q]_\sigma$ sits in cohomological degree $\dim(\sigma)-\dim(\pi\cap-\pi)$ and denotes the localization of $\CC[Q]$ with respect to the multiplicative system $\Set{t^\alpha|\alpha\in \sigma\cap Q}$. The coboundary maps are the natural localization maps with signs chosen appropriately (for details, see \cite[Def.~13.21]{ms05} or \cite[\S2]{ishida}).
\end{definition}

\begin{lemma}\label{lem:open-fsupp}
	If $\beta\notin \mathcal{E}_A$, then $\fSupp\hat{K}^A_\bullet(S_A;E_A-\beta)$ is open in $X_A$.
\end{lemma}
\begin{proof}
	By \Cref{th:orbit-restr}, the orbit-cone correspondence, and \Cref{prop:EA}, it suffices to prove the following: For all faces $F\preceq G \preceq A$,
	\begin{equation*}
		\qdeg H^{d_{A/F}}_{\orbit(F)}(S_A) \subseteq \qdeg H^{d_{A/F}}_{\orbit(G)}(S_A).
	\end{equation*}
	
	To prove this, consider the short exact sequence of complexes
	\begin{equation}\label{eq:ishida-seq}
		0 \to \mho^\bullet_{S_A[\partial^{-F}]}[d_{G/F}] \to \mho^\bullet_{S_A[\partial^{-G}]} \to C^\bullet \to 0,
	\end{equation}
	where the first two complexes are the Ishida complexes of $S_A[\partial^{-F}]$ and $S_A[\partial^{-G}]$, respectively, the first map is the natural inclusion, and the third complex is the cokernel. Since the Ishida complex of $S_F[\partial^{-F}]$ represents $\RGamma_{\orbit(F)}(S_A[\partial^{-F}]) = \RGamma_{\orbit(F)}(S_A)$ (and similarly for $G$), the long exact sequence in cohomology gives an exact sequence
	\[ \HH^{d_{A/F}}_{\orbit(F)}(S_A[\partial^{-F}]) \to \HH^{d_{A/G}}_{\orbit(G)}(S_A[\partial^{-G}]) \to H^{d_{A/G}}(C^\bullet) \to 0.\]
	But the first two complexes in \eqref{eq:ishida-seq} are both equal to $S_A[\partial^{-A}]$ in cohomological degree $d_{A/G}$, so $H^{d_{A/G}}(C^\bullet) = 0$.	Now use that if $M$ is a graded quotient of a graded module $N$, then $\qdeg M\subseteq \qdeg N$.
\end{proof}

\subsection{Contiguity}\label{subsec:contiguity}
In this subsection we discuss how right multiplication by a monomial of $R_A$ (a ``contiguity'' operator) affects the restrictions and exceptional restrictions, respectively, of an Euler--Koszul complex to orbits. 

\begin{lemma}\label{lem:mult-by-alpha}
	Let $F\preceq A$ be a face, and let $M$ be a finitely-generated $\ZZ^d$-graded $S_A$-submodule of $\CC[\ZZ^d]$. Let $\beta\in \CC^d$ and $\alpha\in \NN{A}$. Assume that $\beta,\beta-\alpha\notin \mathcal{E}_{A,F}^{\mathrm{strong}}(M)$. Then the following are equivalent: 
	\begin{enumerate}[label=\textnormal{(\alph*)}]
		\item The morphism $i^+_{\orbit(F)}\hat{K}^A_\bullet(M;E_A-\beta+\alpha)\to i^+_{\orbit(F)}\hat{K}^A_\bullet(M;E_A-\beta)$ induced by right-multiplication by $\partial^\alpha$ is an isomorphism.
		\item\label{item:forall-lambda} For all $\lambda\in \CC{F}$,
		\[ \RGamma_{\orbit(F)}(M)_{\beta-\alpha-\lambda}\neq 0 \text{ if and only if } \RGamma_{\orbit(F)}(M)_{\beta-\lambda}\neq 0.\]
	\end{enumerate}
\end{lemma}
\begin{proof}
	By \Cref{th:orbit-restr} and because neither $\beta$ nor $\beta-\alpha$ are strongly $(A,F)$-exceptional for $M$, it suffices to show that the morphism $f_\lambda \colon H^{d_{A/F}}_{\orbit(F)}(M(-\alpha))_{\beta-\lambda}\to \HH^{d_{A/F}}_{\orbit(F)}(M)_{\beta-\lambda}$ induced by multiplication by $\partial^\alpha$ is an isomorphism for all $\lambda\in \CC{F}$ if and only if \ref{item:forall-lambda}. 
	
	The ``only if'' direction is immediate. For the ``if'' direction, the long exact sequence of local cohomology gives an exact sequence
	\[ \HH^{d_{A/F}}_{\orbit(F)}(M(-\alpha))_{\beta-\lambda}\xrightarrow{\,f_\lambda \,} \HH^{d_{A/F}}_{\orbit(F)}(M)_{\beta-\lambda} \to  \HH^{d_{A/F}}_{\orbit(F)}(M/\partial^\alpha M)_{\beta-\lambda}\to 0.\]
	But $\dim (M/\partial^\alpha M) < d_{A/F}$ because $M$ is finitely generated and $\partial^\alpha$ is $M$-regular. So, $\HH^{d_{A/F}}_{\orbit(F)}(M/\partial^\alpha M)_{\beta-\lambda}=0$, and therefore $f_\lambda$ is always surjective. Moreover, because the Hilbert function of $H_{\orbit(F)}^{d_{A/F}}(M)$ takes values in $\{0,1\}$, the hypothesis \ref{item:forall-lambda} implies that both the domain and codomain of $f_\lambda$ have dimension $1$. Therefore, $f_\lambda$ is an isomorphism for all $\lambda$.
\end{proof}

%
\begin{lemma}\label{lem:om-mult-by-alpha}
	Let $F\preceq A$ be a face, and let $M$ be a finitely-generated $\ZZ^d$-graded $S_A$-submodule of $\CC[\ZZ^d]$. The following are equivalent for $\beta\in \CC^d$ and $\alpha\in \NN{A}$:
	\begin{enumerate}[label=\textnormal{(\alph*)}]
		\item The morphism $i^\dagger_{\orbit(F)}\hat{K}^A_\bullet(M;E_A-\beta)\to i^\dagger_{\orbit(F)}\hat{K}^A_\bullet(M;E_A-\beta-\alpha)$ induced by right-multiplication by $\partial^\alpha$ is an isomorphism.
		\item\label{item:exc-forall-lambda} For all $\lambda\in \CC{F}$,
		\[ M[\partial^{-F}]_{\beta-\lambda}\neq 0 \text{ if and only if }M[\partial^{-F}]_{\beta+\alpha-\lambda}\neq 0.\]
	\end{enumerate}
\end{lemma}
\begin{proof}
	By \Cref{th:exc-restr}, it suffices to show that, as with \Cref{lem:mult-by-alpha}, the morphism $f_\lambda\colon M[\partial^{-F}]_{\beta-\lambda} \to M(\alpha)[\partial^{-F}]_{\beta-\lambda}$ induced by multiplication by $\partial^\alpha$  is an isomorphism for all $\lambda\in \CC{F}$ if and only if \ref{item:exc-forall-lambda}.
	
	As before, the ``only if'' direction is immediate. For the ``if'' direction, $\partial^\alpha$ is $M$- (and therefore $M[\partial^{-F}]$-) regular, so $f_\lambda$ is always injective. Now proceed as in \Cref{lem:mult-by-alpha} using the fact that the Hilbert function of $M[\partial^{-F}]$ takes values in $\{0,1\}$.
\end{proof}

\subsection{Main Theorems}\label{subsec:main-th}

\begin{definition}\label{def:E+dag}
	Given a face $F$ and a parameter $\beta\in \CC^d$, define the sets
	\[ \mathrm{E}^*_F(\beta) \coloneqq \Set{\lambda\in \CC{F}/\ZZ{F} | \RGamma_{\orbit(F)}(S_A)_{\beta-\lambda}\neq 0}\]
	and
	\[ \mathrm{E}_F(\beta) \coloneqq \Set{\lambda \in \CC{F}/\ZZ{F} | S_A[\partial^{-F}]_{\beta-\lambda}\neq 0}.\]
	Because $S_A[\partial^{-F}]\cong \CC\{\NN{A}-\NN{F}\}$, the second set is the set $E_F(\beta)$ defined by Saito in \cite{Sai01}.
\end{definition}

\begin{remark}\label{rmk:supps}
	The definitions of $\mathrm{E}_F(\beta)$ and $\mathrm{E}^*_F(\beta)$ along with \Cref{th:orbit-restr,th:exc-restr} show that for all $\beta\in \CC^d$,
	\begin{align*}
		\fSupp \hat{K}^A_\bullet(S_A; E_A-\beta) &= \bigsqcup_{\mathrm{E}^*_F(\beta)\neq \emptyset} \orbit(F)\\
		\intertext{and}
		\cofSupp \hat{K}^A_\bullet(S_A; E_A-\beta) &= \bigsqcup_{\mathrm{E}_F(\beta)\neq \emptyset} \orbit(F).
	\end{align*}
\end{remark}

\begin{remark}
	Let $\beta\in \CC^d$ and $F\precneqq A$. Suppose $\lambda+\ZZ{F}\in \mathrm{E}_F(\beta)$, so that $\beta-\lambda\in \NN{A}-\NN{F}$. Then $H^i_{\orbit(F)}(S_A)_{\beta-\lambda}$ is isomorphic to the reduced cohomology $\tilde{H}^i(\mathcal{P};\CC)$ of a (nonempty) convex polytope $\mathcal{P}$ (cf.~\cite[Rmk.~13.25 and Cor.~13.26]{ms05}). As convex polytopes are contractible, this cohomology vanishes, and therefore $\lambda+\ZZ F\notin \mathrm{E}_F^*(\beta)$. In other words,
	\[ \mathrm{E}_F(\beta)\cap \mathrm{E}_F^*(\beta) = \emptyset.\]		
\end{remark}

\medskip
Before continuing, we state a small lemma about $\mathrm{E}^*_F(\beta)$ and $\mathrm{E}_F(\beta)$. Parts \ref{lem:stability.EF*} and \ref{lem:stability.EF} follow from \Cref{lem:mult-by-alpha} and \Cref{lem:om-mult-by-alpha}, respectively. Note that \ref{lem:stability.EF} is also \cite[Prop.~2.2 (5)]{Sai01}.

\begin{lemma}\label{lem:stability}
	Let $\beta\in \CC^d$ and $\alpha\in \NN{A}$.
	\begin{enumerate}[label=\textnormal{(\alph*)}]
		\item\label{lem:stability.EF*} If $\beta,\beta-\alpha\notin\mathcal{E}_{A,F}^{\mathrm{strong}}$, then $\mathrm{E}^*_F(\beta)\subseteq \mathrm{E}^*_F(\beta-\alpha)$.
		\item\label{lem:stability.EF} $\mathrm{E}_F(\beta) \subseteq \mathrm{E}_F(\beta+\alpha)$.
	\end{enumerate}
\end{lemma}

\begin{definition}\label{def:mgm}
	\begin{enumerate}
		\item A parameter $\beta\in \CC^d$ is \emph{mixed Gauss--Manin along the face $F\preceq A$} if either $\mathrm{E}_F(\beta)=\emptyset$ or there exists a $\beta'\in \CC^d\setminus \sRes(A)$ with $\beta-\beta'\in \ZZ^d$ such that $\mathrm{E}_F(\beta) = \mathrm{E}_F(\beta')$. A parameter $\beta\in \CC^d$ is \emph{mixed Gauss--Manin} if it is mixed Gauss--Manin along every face.
		
		\item A parameter $\beta\in \CC^d$ is \emph{dual mixed Gauss--Manin along the face $F\preceq A$} if $\beta\notin \mathcal{E}_A$ and if either $\mathrm{E}^*_F(\beta)=\emptyset$ or there exists a $-\beta'\in \CC^d\setminus\sRes(A)$ with $\beta-\beta'\in \ZZ^d$ such that $\mathrm{E}^*_F(\beta) = -\mathrm{E}_F(-\beta')$. A parameter $\beta\in \CC^d$ is \emph{dual mixed Gauss--Manin} if it is dual mixed Gauss--Manin along every face.
	\end{enumerate}
\end{definition}

\begin{remark}
	The proof of \Cref{lem:mult-by-alpha} shows that, at least if $\mathcal{E}_{A,F}^{\mathrm{strong}}=\emptyset$, the condition of being dual mixed Gauss--Manin along $F$ is partially stable in the following sense: If $\beta$ is dual mixed Gauss--Manin along $F$ with $\mathrm{E}^*_F(\beta)\neq\emptyset$, then $\beta-\alpha$ is also dual mixed Gauss--Manin along $F$ for every $\alpha\in \NN{A}$. Similarly, the proof of \Cref{lem:om-mult-by-alpha} shows that if $\beta$ is mixed Gauss--Manin along $F$ with $\mathrm{E}_F(\beta)\neq \emptyset$, then $\beta+\alpha$ is also mixed Gauss--Manin along $F$ for every $\alpha\in \NN{A}$.
\end{remark}

Before stating \Cref{th:plus-dag,th:dag-plus}, we recall the following notation and definitions:
\begin{itemize}
	\item $\fSupp$ and $\cofSupp$ denote fiber support and cofiber support, respectively, and were defined in \eqref{def:fiber-support}. 
	\item $\widehat{\CC^n}\coloneqq \Spec\CC[\partial_1,\ldots,\partial_n]$.
	\item For an open subset $U\subseteq \widehat{\CC^n}$ containing $T_A$, the embeddings $T_A\hookrightarrow U$ and $U\hookrightarrow \widehat{\CC^n}$ are denoted by $\iota_U$ and $\varpi_U$, respectively---these were discussed at the start of \S\ref{sec:dirims}.
\end{itemize}

\begin{theorem}\label{th:plus-dag}
	The following are equivalent for $\beta\in\CC^d$:
	\begin{enumerate}[label=\textnormal{(\alph*)}]
		\item\label{th:plus-dag-cond} $\beta$ is dual mixed Gauss--Manin.
		\item\label{th:plus-dag-some} $K^A_\bullet(S_A; E_A-\beta) \qi \FL(\varpi_{U+}\iota_{U\dagger}\mathcal{O}_{T_A}^\beta)$ for some open subset $U\subseteq \widehat{\CC^n}$ containing $T_A$.
		\item\label{th:plus-dag-any} $K^A_\bullet(S_A; E_A-\beta) \qi \FL(\varpi_{U+}\iota_{U\dagger}\mathcal{O}_{T_A}^\beta)$ for any open subset $U\subseteq \widehat{\CC^n}$ satisfying $U\cap X_A=\fSupp\hat{K}^A_\bullet(S_A; E_A-\beta)$.
	\end{enumerate}
\end{theorem}
\begin{proof} (\ref{th:plus-dag-some}$\implies$\ref{th:plus-dag-cond}) Let $F\preceq A$ be a face. If $\orbit(F)$ is not contained in (hence is disjoint from) $U$, then by \Cref{lem:empty-cap}, the restriction to $\orbit(F)$ of $\varpi_{U+}\iota_{U\dagger}\mathcal{O}_{T_A}^\beta$ vanishes. Therefore, by the hypothesis, the same applies to the restriction to $\orbit(F)$ of $\hat{K}_\bullet^A(S_A;E_A-\beta)$. Hence, \eqref{eq:orbit-restr-cpt} from \Cref{th:orbit-restr} implies that $\RGamma_{\orbit(F)}(S_A)_{\beta+\CC{F}}=0$. In particular, $\beta\notin \mathcal{E}_{A,F}^{\mathrm{strong}}$ and $\mathrm{E}^*_F(\beta)=\emptyset$.
	
	Next, suppose $\orbit(F)\subseteq U$. By \Cref{prop:phi-dagger}, there exists a $\beta'$, which may be chosen such that $-\beta'$ is not strongly resonant (cf.~\cite[the discussion preceding Cor.~3.9]{SW09}), with $\beta-\beta'\in \NN{A}$ and such that $\varphi_\dagger\mathcal{O}_{T_A}^\beta$ is isomorphic to $\hat{K}^A_\bullet(S_A; E_A-\beta')$. We fix such a $\beta'$. By \Cref{th:orbit-restr},
	\begin{equation}\label{eq:SA-restr}
		i_{\orbit(F)}^+\hat{K}_\bullet^A(S_A; E_A-\beta) \cong \bigoplus_{\lambda+\ZZ{F}\in \CC{F}/\ZZ{F}} \mathcal{O}_{T_F}^\lambda \otimes_\CC \RGamma_{\orbit(F)}(S_A)_{\beta-\lambda} \otimes_\CC \Bigwedge\CC^{d_{A/F}},
	\end{equation}
	By \Cref{th:orbit-restr} together with \Cref{lem:dag-plus},
	\begin{align}
		i_{\orbit(F)}^+\varpi_{U+}\iota_{U\dagger}\mathcal{O}_{T_A}^\beta
		&\cong i_{\orbit(F)}^+\RGamma_U(\varphi_\dagger\mathcal{O}_{T_A}^\beta)\notag\\
		&\cong i_{\orbit(F)}^+\varphi_\dagger\mathcal{O}_{T_A}^\beta\notag\\
		&\cong i_{\orbit(F)}^+\hat{K}^A_\bullet(S_A; E_A-\beta')\notag\\
		&\cong \bigoplus_{\lambda+\ZZ{F}\in \CC{F}/\ZZ{F}} \mathcal{O}_{T_F}^\lambda \otimes_\CC \RGamma_{\orbit(F)}(\omega_{S_A}^\bullet)_{\beta'-\lambda} \otimes_\CC \Bigwedge\CC^{d_{A/F}}.\label{eq:dagger-restr}
	\end{align}
	The left hand sides of \eqref{eq:SA-restr} and \eqref{eq:dagger-restr} are quasi-isomorphic by hypothesis. Hence, the same is true of the right hand sides of \eqref{eq:SA-restr} and \eqref{eq:dagger-restr}---call this isomorphism $\psi$. Now, the modules $\mathcal{O}_{T_A}^\lambda$ are simple of different weights, and the differentials of $\Bigwedge\CC^{d_{A/F}}$ are all $0$. Therefore, $\psi$ induces a quasi-isomorphism between $\RGamma_{\orbit(F)}(S_A)_{\beta-\lambda}$ and $\RGamma_{\orbit(F)}(\omega_{S_A}^\bullet)_{\beta'-\lambda}$ for all $\lambda\in \CC{F}$. But by \eqref{eq:face-LC-om}, we know that $\RGamma_{\orbit(F)}(\omega_{S_A}^\bullet)\cong \CC\{\NN{F}-\NN{A}\}[-d_{A/F}]$. Hence, $\RGamma_{\orbit(F)}(S_A)_{\beta-\lambda}$ can have cohomology only in cohomological degree $d_{A/F}$ and is nonzero if and only if $\lambda+\ZZ{F}\in -\mathrm{E}_F(-\beta')$. Thus, $\beta$ is not strongly $(A,F)$-resonant, and $\mathrm{E}^*_F(\beta) = -\mathrm{E}_F(-\beta')$. Now use \Cref{prop:EA} and \Cref{def:mgm}.
	
	(\ref{th:plus-dag-cond}$\implies$\ref{th:plus-dag-any}) Let $\beta'$ be as above. Consider the morphism
	\[\eta=\cdot\partial^{\beta-\beta'}\colon \hat{K}^A_\bullet(S_A;E_A-\beta') \to \hat{K}^A_\bullet(S_A; E_A-\beta).\]
	Let $U$ be an open subset of $\widehat{\CC^n}$ with $U\cap X_A=\fSupp\hat{K}^A_\bullet(S_A;E_A-\beta)$; such a $U$ exists by \Cref{lem:open-fsupp}. Then $\RGamma_{\widehat{\CC^n}\setminus U} \hat{K}^A_\bullet(S_A; E_A-\beta)$ vanishes by \Cref{prop:lc-vs-fSupp}. So, from the distinguished triangle relating $\RGamma_{U}$ and $\RGamma_{\widehat{\CC^n}\setminus U}$, we get that $\RGamma_{U}\hat{K}^A_\bullet(S_A; E_A-\beta)$ is isomorphic to $\hat{K}^A_\bullet(S_A; E_A-\beta)$. Thus, it remains to show that $\RGamma_U(\eta)$ is an isomorphism.
	
	Now, $\RGamma_U(\eta)$ is an isomorphism if and only if its cone vanishes, and cones commute with $\RGamma_U$, so we need to show that $\RGamma_U(\cone\eta)=0$. By \Cref{prop:lc-vs-fSupp}, this is true if and only if the fiber support of $\cone\eta$ is disjoint from $U$. So, we just need to show that $i^+_{\orbit(F)}\cone\eta=0$ for all $\orbit(F)\subseteq U$. Pulling out the cone, we just need to show that $\cone (i^+_{\orbit(F)}\eta)=0$ for all $\orbit(F)\subseteq U$, i.e.~that $i^+_{\orbit(F)}\eta$ is an isomorphism for all $\orbit(F)\subseteq U$. This is true by \Cref{lem:mult-by-alpha}.
	
	(\ref{th:plus-dag-any}$\implies$\ref{th:plus-dag-some}) Immediate.
\end{proof}

\begin{remark}\label{rmk:E+vsEdag}
	Let $\beta\in \CC^d$ with $\varphi_\dagger\mathcal{O}_{T_A}\cong \hat{K}^A_\bullet(S_A; E_A-\beta)$. Then the proof of (a$\implies$b) in \Cref{th:plus-dag} shows that $\beta\notin \mathcal{E}_A$, and $\mathrm{E}^*_F(\beta)=-\mathrm{E}_F(-\beta)$ for all $F\preceq A$.
\end{remark}

\begin{theorem}\label{th:dag-plus}
	The following are equivalent for $\beta\in\CC^d$:
	\begin{enumerate}[label=\textnormal{(\alph*)}]
		\item\label{th:dag-plus-cond} $\beta$ is mixed Gauss--Manin.
		\item\label{th:dag-plus-some} $K^A_\bullet(S_A; E_A-\beta) \qi \FL(\varpi_{U\dagger}\iota_{U+}\mathcal{O}_{T_A}^\beta)$ for some open subset $U\subseteq \widehat{\CC^n}$ containing $T_A$.
		\item\label{th:dag-plus-any} $K^A_\bullet(S_A; E_A-\beta) \qi \FL(\varpi_{U\dagger}\iota_{U+}\mathcal{O}_{T_A}^\beta)$ for any open subset $U\subseteq \widehat{\CC^n}$ satisfying $U\cap X_A = \cofSupp\hat{K}^A_\bullet(S_A; E_A-\beta)$.
	\end{enumerate}
\end{theorem}
\begin{proof}(\ref{th:dag-plus-some}$\implies$\ref{th:dag-plus-cond}) Let $F\preceq A$ be a face. If $\orbit(F)$ is not contained in (hence disjoint from) $U$, then $i_{\orbit(F)}^\dagger\varpi_{U\dagger}\iota_U^+\mathcal{O}_{T_A}^\beta = 0$ by \Cref{lem:empty-cap}. So, $\mathrm{E}_F(\beta)=\emptyset$ by the hypothesis and \Cref{th:exc-restr}.
	
	Next, suppose $\orbit(F)$ is contained in $U$. Choose a $\beta'\in \beta + \ZZ^d$ which is not strongly resonant. Then by \Cref{th:exc-restr},
	\begin{equation}\label{eq:exc-SA-restr}
		i^\dagger_{\orbit(F)}\hat{K}^A_\bullet(S_A; E_A-\beta) \cong \bigoplus_{\lambda+\ZZ{F}\in \CC{F}/\ZZ{F}} \mathcal{O}_{T_F}^\lambda \otimes_\CC S_A[\partial^{-F}]_{\beta-\lambda}\otimes_\CC \Bigwedge\CC^{d_{A/F}},
	\end{equation}
	and by \Cref{th:exc-restr} together with \cite[Cor.~3.7]{SW09} and \Cref{lem:dag-plus},
	\begin{align}
		i^\dagger_{\orbit(F)}\varpi_{U\dagger}\iota_{U+}\mathcal{O}_{T_A}^\beta
		&\cong i^\dagger_{\orbit(F)} \varpi_{U\dagger}\varpi_U^{-1}\varphi_+\mathcal{O}_{T_A}^\beta\notag\\
		&\cong i^\dagger_{\orbit(F)} \varphi_+\mathcal{O}_{T_A}^\beta\notag\\
		&\cong i^\dagger_{\orbit(F)} \hat{K}^A_\bullet(S_A; E_A-\beta')\notag\\
		&\cong \bigoplus_{\lambda+\ZZ{F}\in \CC{F}/\ZZ{F}} \mathcal{O}_{T_F}^\lambda \otimes_\CC S_A[\partial^{-F}]_{\beta'-\lambda}\otimes_\CC \Bigwedge\CC^{d_{A/F}}.\label{eq:exc-dagger-restr}
	\end{align}
	The left hand sides of \eqref{eq:exc-SA-restr} and \eqref{eq:exc-dagger-restr} are isomorphic by hypothesis. Hence, the same is true of the right hand sides---call this isomorphism $\psi$. As in the proof of \Cref{th:plus-dag}, the modules $\mathcal{O}_{T_A}^\lambda$ are simple of different weights, and the differentials of $\Bigwedge\CC^{d_{A/F}}$ are all $0$. Therefore, $\psi$ induces an isomorphism between $S_A[\partial^{-F}]_{\beta-\lambda}$ and $S_A[\partial^{-F}]_{\beta'-\lambda}$. Now use the definition of $\mathrm{E}_F$.
	
	(\ref{th:dag-plus-cond}$\implies$\ref{th:dag-plus-any}) Let $\beta'$ be as above. Consider the morphism
	\[ \eta = \cdot \partial^{\beta'-\beta} \colon \hat{K}^A_\bullet(S_A; E_A-\beta') \to \hat{K}_\bullet^A(S_A; E_A-\beta).\]
	Let $U$ be a Zariski open subset of $\widehat{\CC^n}$ with $U\cap X_A = \cofSupp \hat{K}^A_\bullet(S_A; E_A-\beta)$; such a $U$ exists by \cite[Prop.~2.2 (4)]{Sai01} and the orbit-cone correspondence. Now use the same argument as in the proof of \Cref{th:plus-dag} with $\Ddual \hat{K}^A_\bullet(S_A; E_A-\beta)$, $\Ddual \eta$, and \Cref{lem:om-mult-by-alpha} in place of $\hat{K}^A_\bullet(S_A; E_A-\beta)$, $\eta$, and \Cref{lem:mult-by-alpha}, respectively.
	
	(\ref{th:dag-plus-any}$\implies$\ref{th:dag-plus-some}) Immediate.
\end{proof}

The following example shows that in general, not every $\beta$ is mixed or dual mixed Gauss--Manin even if $S_A$ is Cohen--Macaulay.

\begin{example}\label{ex:8isoms}
	Let
	\[ A = \begin{bmatrix}
		1 &1 &0\\
		0 &1 &2
	\end{bmatrix}.\]
	The associated semigroup ring $S_A$ is Cohen--Macaulay but not normal. For simplicity, we only discuss $\beta\in \ZZ^2$. There are 8 isomorphism classes---these are pictured in \Cref{fig:8isoms}.
	\begin{figure}
		\centering
		\includegraphics{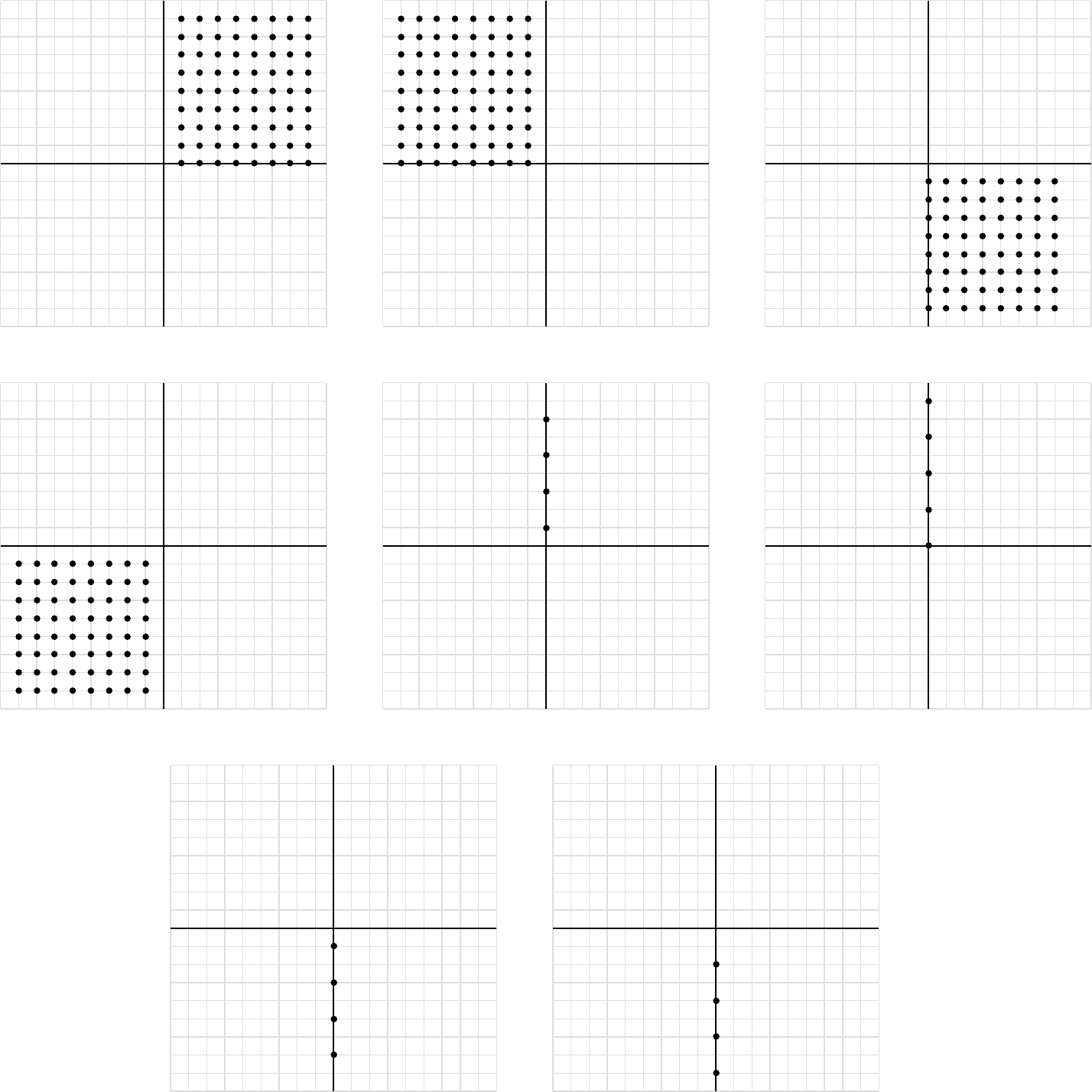}
		\caption{The eight integral isomorphism classes from \Cref{ex:8isoms}.}
		\label{fig:8isoms}
	\end{figure}
	Of these, only the first four (numbered from left to right then top to bottom) are mixed Gauss--Manin, and only these first four are dual mixed Gauss--Manin. The fiber supports of the 8 classes are, in order,
	\begin{gather*}
		\orbit(A),\quad \orbit(A)\cup\orbit([\vec{a}_3]),\quad \orbit(A)\cup\orbit([\vec{a}_1]),\quad X_A,\\
		\orbit(A)\cup \orbit([\vec{a}_3]), \quad \orbit(A)\cup \orbit([\vec{a}_3]), \quad X_A, \quad X_A.
	\end{gather*}
	The cofiber supports of the 8 classes are, in order,
	\begin{gather*}
		X_A,\quad \orbit(A)\cup\orbit([\vec{a}_1]),\quad  \orbit(A)\cup \orbit([\vec{a}_3]), \quad \orbit(A)\\
		\orbit(A)\cup\orbit([\vec{a}_1])\cup\orbit([\vec{a}_3]), \quad \orbit(A)\cup\orbit([\vec{a}_1])\cup\orbit([\vec{a}_3])\\
		\orbit(A)\cup \orbit([\vec{a}_3]), \quad \orbit(A)\cup \orbit([\vec{a}_3]).
	\end{gather*}
	The fiber supports were computed using Macaulay2 (\cite{M2}) by restricting the various modules $\hat{M}_A(\beta)$ to the various distinguished points $\mathbbm{1}_F$ and asking whether or not the result vanished. To compute the cofiber supports, we implemented \cite[Algorithms 3.4.2 and 3.4.3]{ST01} in Macaulay2.
\end{example}

%

\section{Normal Case}
In this section, we prove (\Cref{th:normal}) that if $S_A$ is normal, then every parameter $\beta$ is both mixed Gauss--Manin and dual mixed Gauss--Manin. \Cref{lem:normal} provides an explicit description of the fiber and cofiber supports of $\hat{M}_A(\beta)$ and computes the restrictions of $\hat{M}_A(\beta)$ to the various orbits. In a future paper, we will apply \Cref{th:normal} to compute for such $A$ the projection and restriction of $M_A(\beta)$ to coordinate subspaces of the form $\CC^F$, where $F$ is a face of $A$; and, if $A$ is in addition homogeneous, to show that the holonomic dual of $M_A(\beta)$ is itself $A$-hypergeometric.

\medskip
Recall that for a facet $G\preceq \NN{A}$, there is a unique linear form $h_G\colon \ZZ^d\to \ZZ$, called the \emph{primitive integral support function} of $G$, satisfying the following conditions:
\begin{enumerate}
	\item $h_G(\ZZ^d)=\ZZ$.
	\item $h_G(\vec{a}_i)\geq0$ for all $i$.
	\item $h_G(\vec{a}_i)=0$ for all $\vec{a}_i\in G$.
\end{enumerate}

\begin{lemma}\label{lem:normal}
	Assume $S_A$ is normal. Let $\beta\in \CC^d$ and $F\preceq A$.
	\begin{enumerate}[label=\textnormal{(\alph*)}]
		\item\label{lem:normal.restr} $i^+_{\orbit(F)}\hat{M}_A(\beta)$ is either zero or isomorphic to $\mathcal{O}_{T_F}^\lambda \otimes_\CC\Bigwedge\CC^{d_{A/F}}[-d_{A/F}]$ for some (equiv.\ any) $\lambda\in \CC{F}$ with $\beta-\lambda\in \ZZ^d$.
		\item\label{lem:normal.exc-restr} $i^\dagger_{\orbit(F)}\hat{M}_A(\beta)$ is either zero or isomorphic to $\mathcal{O}_{T_F}^\lambda \otimes_\CC\Bigwedge\CC^{d_{A/F}}$ for some (equiv.\ any) $\lambda\in \CC{F}$ with $\beta-\lambda\in \ZZ^d$.
		\item\label{lem:normal.orbit-fsupp} $\orbit(F)\subseteq \fSupp\hat{M}_A(\beta)$ if and only if $(\beta+\CC{F})\cap \ZZ^d\neq\emptyset$ and $h_G(\beta)\in\ZZ_{<0}$ for every facet $G\succeq F$.
		\item\label{lem:normal.orbit-cofsupp} $\orbit(F)\subseteq \cofSupp\hat{M}_A(\beta)$ if and only if $(\beta+\CC{F})\cap \ZZ^d\neq\emptyset$ and $h_G(\beta)\in \NN$ for every facet $G\succeq F$.
	\end{enumerate}
\end{lemma}
\begin{proof}
	Before proving the statements, notice that because $S_A$ is normal, it is Cohen--Macaulay by \cite[Theorem~1]{hoch72}. Therefore, $\hat{M}_A(\beta) \qi \hat{K}_\bullet^A(S_A; E_A-\beta)$ by \cite[Th.~6.6]{MMW05}.
	
	\ref{lem:normal.restr} Since $S_A$ is Cohen--Macaulay, the complex $\RGamma_{\orbit(F)}(S_A)$ has cohomology only in cohomological degree $d_{A/F}$, so that
	\begin{equation}\label{eq:EF*normal}
		\mathrm{E}^*_F(\beta) = \Set{\lambda + \ZZ{F}\in \CC{F}/\ZZ{F}|  \HH^{d_{A/F}}_{\orbit(F)}(S_A)_{\beta-\lambda} \neq 0}.
	\end{equation}
	Suppose $\lambda+\ZZ{F},\lambda'+\ZZ{F}\in \mathrm{E}^*_F(\beta)$.  Then $\lambda$ and $\lambda'$ differ by an element $\CC{F}\cap \ZZ^d$. But by normality, $\CC{F}\cap \ZZ^d=\ZZ{F}$. Hence, $\lambda+\ZZ{F}=\lambda'+\ZZ{F}$. Now apply \Cref{th:orbit-restr}, and use \eqref{eq:EF*normal} along with the fact that the Hilbert function of $\HH^{d_{A/F}}_{\orbit(F)}(S_A)$ takes values in $\{0,1\}$.
	
	\ref{lem:normal.exc-restr} As in \ref{lem:normal.restr}, normality implies that $\mathrm{E}_F(\beta)$ has at most one element (this also follows from \cite[Prop.~2.3 (1)]{Sai01}). Now apply \Cref{th:exc-restr} along with the fact that the Hilbert function of $S_A[\partial^{-F}]$ takes values in $\{0,1\}$.
	
	\ref{lem:normal.orbit-fsupp} 
	
	By \Cref{th:orbit-restr}, we need to show that $\RGamma_{\orbit(F)}(S_A)_{\beta+\CC{F}}\neq 0$ if and only if $h_G(\beta)\in \ZZ_{<0}$ for all facets $G\succeq F$. As in \ref{lem:normal.restr}, $\RGamma_{\orbit(F)}(S_A)$ is concentrated in cohomological degree $d_{A/F}$. Since $S_A$ is normal, $\HH^{d_{A/F}}_{\orbit(F)}(S_A)=\CC\{-\relint(\NN{A}-\NN{F})\}$, where $\relint$ denotes the relative interior of an affine semigroup (i.e.~the set of points in the affine semigroup which are not on any of its facets). In terms of the primitive integral support functions, $-\relint(\NN{A}-\NN{F})$ consists of those points $\alpha\in \ZZ^d$ such that $h_G(\alpha)<0$ for all facets of $A$ which contain $F$. Thus, $\RGamma_{\orbit(F)}(S_A)_{\beta+\CC{F}}\neq 0$ if and only if there exists a $\lambda\in \CC{F}$ with $\beta-\lambda\in \ZZ^d$ such that $h_G(\beta-\lambda)\in \ZZ_{<0}$ for all facets $G\succeq F$. But $h_G$ kills $\CC{F}$ by definition, and $\beta+\CC^d$ intersects $\ZZ^d$ by assumption. So, $\RGamma_{\orbit(F)}(S_A)_{\beta+\CC{F}}\neq 0$ if and only if $h_G(\beta)\in \ZZ_{<0}$ for all facets $G\succeq F$.
	
	\ref{lem:normal.orbit-cofsupp} The proof of \cite[Th.~5.2]{Sai01} shows that $\mathrm{E}_F(\beta)$ is non-empty if and only if $(\beta+\CC{F})\cap \ZZ^d\neq \emptyset$ and $h_G(\beta)\in \NN$ for every facet $G\succeq F$. Now use \Cref{th:exc-restr}.
\end{proof}

The condition $(\beta+\CC{F})\cap\ZZ^d\neq \emptyset$ in \Cref{lem:normal}\ref{lem:normal.orbit-fsupp} and \ref{lem:normal.orbit-cofsupp} is necessary, as the following example shows:

\begin{example}\label{ex:ZZd-plus-CCF}
	\begin{figure}
		\centering
		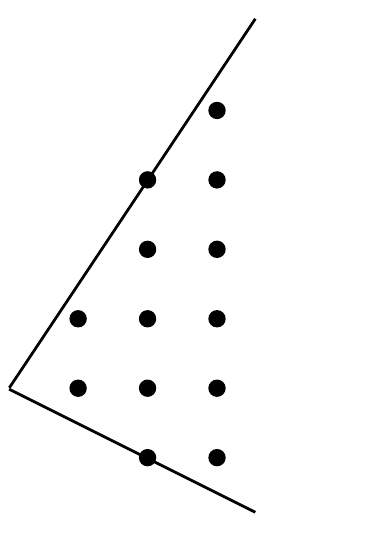
		\caption{The affine semigroup from \Cref{ex:ZZd-plus-CCF}.}\label{fig:ZZd-plus-CCF}
	\end{figure}
	Choose a matrix $A$ generating the affine semigroup pictured in \Cref{fig:ZZd-plus-CCF}. As in the figure, denote by $G_1$ and $G_2$, respectively, the facets $\begin{bmatrix}2&3\end{bmatrix}^\top$ and $\begin{bmatrix}2&-1\end{bmatrix}^\top$ of $A$. Then $h_{G_1}=\begin{bmatrix}3&-2\end{bmatrix}$ and $h_{G_2}=\begin{bmatrix}1&2\end{bmatrix}$.
	
	Consider the parameter $\beta=(-1,-1/2)$. Then $h_{G_1}(\beta)=-3-2(-1/2)=-2\in \ZZ_{<0}$ and $h_{G_2}(\beta)=-1+2(-1/2)=-2\in \ZZ_{<0}$, so by \Cref{lem:normal}\ref{lem:normal.orbit-fsupp}, the fiber support of $\hat{M}_A(\beta)$ contains both $\orbit(G_1)$ and $\orbit(G_2)$. But $\beta + \CC\emptyset = \beta\notin \ZZ^2$, so by \Cref{lem:normal}\ref{lem:normal.orbit-fsupp}, the fiber support of $\hat{M}_A(\beta)$ does not contain $\orbit(\emptyset)$. 
	
	To see the necessity of the condition for \Cref{lem:normal}\ref{lem:normal.orbit-cofsupp}, use a similar argument with the same $A$ for $\beta=(1,1/2)$.
\end{example}

\begin{theorem}\label{th:normal}
	Assume $S_A$ is normal. Let $\beta\in \CC^d$, let $U\subseteq \widehat{\CC^n}$ be an open subset with $U\cap X_A=\fSupp\hat{M}_A(\beta)$, and let $V\subseteq \widehat{\CC^n}$ be an open subset with $V\cap X_A=\cofSupp\hat{M}_A(\beta)$. Then
	\[ \FL(\varpi_{V\dagger}\iota_{V+}\mathcal{O}_{T_A}^\beta) \cong M_A(\beta)\cong \FL(\varpi_{U+}\iota_{U\dagger}\mathcal{O}_{T_A}^\beta).\]
\end{theorem}
\begin{proof}
	As in \Cref{lem:normal}, $\hat{M}_A(\beta) \qi \hat{K}_\bullet^A(S_A; E_A-\beta)$. By \cite[Th.~6.6]{MMW05}, this implies that $\mathcal{E}_A=\emptyset$.

	To prove the first isomorphism, choose a $\beta'\in \CC^d\setminus \sRes(A)$ with $\beta'-\beta \in \NN{A}$ (this is always possible---see \cite[the discussion preceding Cor.~3.9]{SW09}). Let $F\preceq A$ be a face. By \Cref{lem:stability}\ref{lem:stability.EF}, we have $\mathrm{E}_F(\beta)\subseteq \mathrm{E}_F(\beta')$, and by \Cref{lem:normal}\ref{lem:normal.exc-restr}, both $\mathrm{E}_F(\beta)$ and $\mathrm{E}_F(\beta')$ consist of at most one element. Therefore, if $\mathrm{E}_F(\beta)$ is non-empty, then it equals $\mathrm{E}_F(\beta')$. Hence, $\beta$ is mixed Gauss--Manin along $F$. Thus, $\beta$ is mixed Gauss--Manin.
	
	We now prove the second isomorphism. As in the proof of (\ref{th:plus-dag-some}$\implies$\ref{th:plus-dag-cond}) in \Cref{th:plus-dag}, choose a $-\beta'\in \CC^d\setminus \sRes(A)$ with $\beta-\beta'\in \NN{A}$ such that $\varphi_\dagger\mathcal{O}_{T_A}^\beta$ is isomorphic to $\hat{M}_A(\beta')$. Now proceed as for the first isomorphism, using \Cref{lem:stability}\ref{lem:stability.EF*}, $\mathrm{E}^*_F$, and \Cref{lem:normal}\ref{lem:normal.restr} in place of \Cref{lem:stability}\ref{lem:stability.EF}, $\mathrm{E}_F$, and \Cref{lem:normal}\ref{lem:normal.exc-restr}, respectively.
\end{proof}

\begin{example}\label{ex:111-012}
	Let
	\[ A = \begin{bmatrix}
		1& 1& 1\\
		0& 1& 2
	\end{bmatrix}.\]
	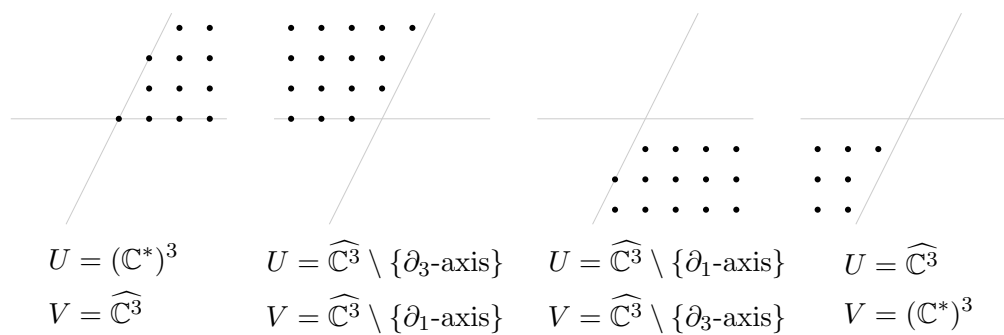
\begin{figure}
		\centering
		\begin{tikzpicture}
			\node[anchor=south west,inner sep=0] (image) at (0,0) {\includegraphics{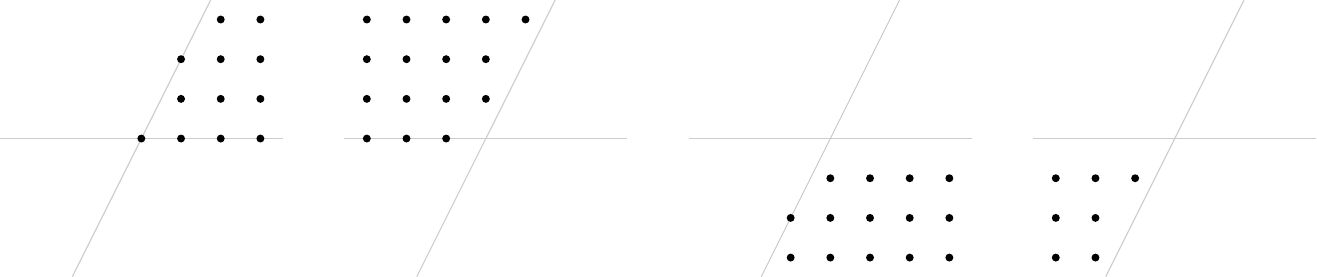}};
			\begin{scope}[x={(image.south east)},y={(image.north west)}]
				\node [anchor=center] at (.1,-.3) {\small $\begin{aligned}U&=(\CC^*)^3\\ V&=\widehat{\CC^3}\end{aligned}$};
				\node [anchor=center] at (.37,-.3) {\small $\begin{aligned}U&=\widehat{\CC^3}\setminus \{\partial_3\text{-axis}\}\\ V&=\widehat{\CC^3}\setminus \{\partial_1\text{-axis}\}\end{aligned}$};
				\node [anchor=center] at (.65,-.3) {\small $\begin{aligned}U&=\widehat{\CC^3}\setminus \{\partial_1\text{-axis}\}\\ V&=\widehat{\CC^3}\setminus \{\partial_3\text{-axis}\}\end{aligned}$};
				\node [anchor=center] at (.89,-.3) {\small $\begin{aligned}U&=\widehat{\CC^3}\\ V&=(\CC^*)^3\end{aligned}$};
			\end{scope}
		\end{tikzpicture}
		\caption{The four isomorphism classes from \Cref{ex:111-012}. The lines are the spans of the two facets of $\RR_{\geq 0}A$.}\label{fig:111-012}
	\end{figure}
	The associated seimgroup ring $S_A$ is a normal. For simplicity, we only discuss $\beta\in \ZZ^2$. There are four isomorphism classes of $A$-hypergeometric systems with $\beta\in\ZZ^2$; these are pictured in \Cref{fig:111-012} along with a $U$ and a $V$ as in \Cref{th:normal}. We now explain why these $U$ and $V$ work by computing the fiber and cofiber supports, using \Cref{lem:normal}, for each of the four isomorphism classes.
	
	The primitive integral support function corresponding to the facets $[\vec{a}_1]$ and $[\vec{a}_3]$ are $h_1(t_1,t_2)=t_2$ and $h_2(t_1,t_2)=2t_1-t_2$, respectively. If $M_A(\beta)$ is in the first (counted from left to right in \Cref{fig:111-012}) isomorphism class, then $h_1(\beta)$ and $h_2(\beta)$ are both in $\NN$, so $\fSupp{\hat{M}_A(\beta)}=\orbit(A)$ and $\cofSupp\hat{M}_A(\beta)=X_A$. If $M_A(\beta)$ is in the second class, then $h_1(\beta)\in\NN$ and $h_2(\beta)\in \ZZ_{<0}$, so $\fSupp{\hat{M}_A(\beta)}=\orbit(A)\cup\orbit([\vec{a}_3])$ and $\cofSupp\hat{M}_A(\beta)=\orbit(A)\cup\orbit([\vec{a}_1])$. If $M_A(\beta)$ is in the third class, then $h_1(\beta)\in\ZZ_{<0}$ and $h_2(\beta)\in \NN$, so $\fSupp{\hat{M}_A(\beta)}=\orbit(A)\cup\orbit([\vec{a}_1])$ and $\cofSupp\hat{M}_A(\beta)=\orbit(A)\cup\orbit([\vec{a}_3])$. If $M_A(\beta)$ is in the fourth class, then $h_1(\beta)$ and $h_2(\beta)$ are both in $\ZZ_{<0}$, so $\fSupp{\hat{M}_A(\beta)}=X_A$ and $\cofSupp\hat{M}_A(\beta)=\orbit(A)$.
\end{example}

\medskip

\bibliographystyle{amsalpha}
\bibliography{gkz}

\end{document}